%% file: NS_stabil_arx.tex
\documentclass[a4paper,reqno]{amsart}
\usepackage{amscd,amsthm,amsmath,amssymb,mathtools}
\usepackage{upgreek,bm}
\usepackage{verbatim}
\usepackage{xcolor}
\usepackage{url}
\usepackage{enumerate,enumitem}
\usepackage{graphicx}
\usepackage{epstopdf,epsfig,subfigure}
\usepackage{curve2e}
\usepackage{array}

\usepackage{algorithm}
\usepackage{algorithmic}

\usepackage[numbers,sort&compress]{natbib}

\usepackage{setspace}
\usepackage[
    colorlinks=true,       
    linkcolor=red,          
    citecolor=blue,        
    filecolor=magenta,      
    urlcolor=cyan,           
]{hyperref}

\theoremstyle{plain}
\newtheorem{theorem}{Theorem}[section]
 \newtheorem{mainresult}[theorem]{Main Result}
 \newtheorem{maincorollary}[theorem]{Main Corollary}

\newtheorem{lemma}[theorem]{Lemma}

\newtheorem{assumption}[theorem]{Assumption}

\theoremstyle{definition}

\newtheorem{remark}[theorem]{Remark}
\newtheorem{example}[theorem]{Example}

\numberwithin{equation}{section}

\usepackage{geometry}

\input{Mathcommands}

\begin{document}
\title{Stabilization of 2D Navier--Stokes equations by means of actuators with locally supported vorticity}
\author{S\'ergio S.~Rodrigues$^{\tt1}$}
\author{Dagmawi A. Seifu$^{\tt1}$}
\thanks{
\vspace{-1em}\newline\noindent
\newline\noindent
$^{\tt1}$ Johann Radon Institute for Computational and Applied Mathematics,
  \"OAW, Altenbergerstr. 69, 4040 Linz, Austria.
  \newline\noindent
{\sc Emails}:
{\small\tt sergio.rodrigues@ricam.oeaw.ac.at,\quad
dagmawi.seifu@ricam.oeaw.ac.at}%
}

\begin{abstract}
Exponential stabilization to time-dependent trajectories for the incompressible Navier--Stokes equations is achieved with explicit feedback controls. The fluid is contained in two-dimensional spatial domains and the control force is, at each time instant, a linear combination of a finite number of given actuators. Each actuator has its vorticity supported in a small subdomain. The velocity field is subject to Lions boundary conditions.
Simulations
are presented showing the stabilizing performance of the
proposed feedback.
The results also apply to a class of observer design problems.
\end{abstract}

\maketitle

\pagestyle{myheadings} \thispagestyle{plain} \markboth{\sc  S. S.
Rodrigues and D. A. Seifu}{\sc Stabilization of Navier--Stokes equations}

\section{Introduction}
Let us be given a trajectory $y_\ttt$ of the Navier--Stokes system as
\begin{subequations}\label{sys-NS}
  \begin{align}
 &\tfrac{\p}{\p t}y_\ttt -  \nu\Delta y_\ttt  +\langle y_\ttt\cdot\nabla\rangle y_\ttt + \nabla p_{\tt t}=f,\qquad\diver y_\ttt = 0,\\
 & \clG y_\ttt\rest{\p\Omega}= 0,\qquad y_\ttt(0,\Bigcdot)= y_{\ttt0}.
\end{align}
\end{subequations}
The spatial domain~$\Omega\in\bbR^2$, is a bounded convex polygonal domain. Denoting by~$(t,x)$ a generic point in the cylinder~$(0,+\infty)\times\Omega$, the state~$y_\ttt=y_\ttt(t,x)\in\bbR^2$ is the velocity vector, ~$p_{\tt t}=p_{\tt t}(t,x)\in\bbR$ represents the pressure function, $f=f(t,x)\in\bbR^2$ is an given external body force, and~$y_{\ttt0}=y_{\ttt0}(x)\in\bbR^2$ is a given initial velocity field, at time~$t=0$. The divergence free relation 
\begin{equation}\notag
0=\diver y_\ttt \eqqcolon  \tfrac{\p}{\p x_1}{y_\ttt}_1+\tfrac{\p}{\p x_2}{y_\ttt}_2
\end{equation}
means that we consider incompressible fluids. Above,~$({y_\ttt}_1,{y_\ttt}_2)\coloneqq y_\ttt$ and~$(x_1,x_2)\coloneqq x$ denote the coordinates of the state~$y_\ttt$ and of the spatial point~$x\in\Omega$. The operators~$\Delta$ and~$\nabla$ stand for the usual Laplacian and gradient operators, formally, for a scalar function~$g=g(x)$,
\begin{equation}\notag
\Delta g \coloneqq (\tfrac{\p}{\p x_1})^2g+ (\tfrac{\p}{\p x_2})^2g,\qquad \nabla g\coloneqq\left(\tfrac{\p}{\p x_1}g,\tfrac{\p}{\p x_2}g\right)
\end{equation}
and for vector fields~$y= (y_1, y_2)$ and~$z= (z_1,z_2)$,
\begin{equation}\notag
\Delta y=(\Delta y_1,\Delta y_2)\qquad\langle  y\cdot\nabla\rangle z \coloneqq  \left(y\cdot \nabla z_1,y\cdot\nabla z_2\right).
\end{equation}
Finally, the relation~$\clG y_\ttt\rest{\p\Omega}= 0$ stands for the (homogeneous) boundary conditions of the fluid velocity~$y_\ttt$.
We shall assume Lions boundary conditions. Namely, firstly the velocity is tangent to the boundary, that is, $(y_\ttt\cdot \bfn)\rest{\p\Omega}=0$ where~$\bfn$ stands for the unit outward normal vector to the boundary~$\p\Omega$ of~$\Omega$. We complement this with a  condition for the vorticity function
\begin{equation}\label{curl}
w_\ttt\coloneqq\curl y_\ttt\coloneqq \tfrac{\p}{\p x_2}{y_\ttt}_1-\tfrac{\p}{\p x_1}{y_\ttt}_2.
\end{equation}
That is, Lions boundary conditions correspond to
\begin{align}\label{lions-bc}
\clG y_\ttt\coloneqq(y_\ttt\cdot \bfn, \curl y_\ttt).
\end{align}
The terminology is motivated by~\cite[Sect.~6.9]{Lions69} and is also adopted in~\cite{Kelliher06,PhanRod17}. It  distinguishes Lions boundary conditions as a subclass of the more general class of Navier (slip) boundary conditions~\cite[Cor.~4.3]{Kelliher06}, which are appropriate to model the fluid velocity in some situations~\cite{ChemetovCiprianoGavrilyuk10}; for further works addressing these conditions we refer to~\cite{CloMikRob98,Coron96,BVeigaCrispo10,BerselliSpirito12,AmroucheRejaiba14} and references therein.

\subsection{Stabilization to trajectories.}
In real world applications we will likely have at our disposal a finite number of actuators only. We take this into consideration throughout this manuscript by taking a finite number~$M_\sigma$ of vector fields~$\Phi_j=\Phi_j(x)$, $1\le j\le M_\sigma$, as actuators.
Given a different initial state~$y_0\ne{y_{\tt t0}}$ we want to find a control input~$u=u(t)\in\bbR^{M_\sigma}$ such that the solution~$y$ of the controlled Navier--Stokes system
\begin{subequations}\label{sys-NS-control}
\begin{align}
 &\tfrac{\p}{\p t} y -  \nu\Delta y  +\langle  y\cdot\nabla\rangle  y + \nabla p=f+\sum_{j=1}^{M_\sigma}u_j\Phi_j,\qquad \diver y = 0,\\
 &\clG y\rest{\p\Omega}= 0,\qquad y(0,\Bigcdot)= y_0.
\end{align} 
\end{subequations}
converges to the targeted solution~$y_\ttt$ as time increases.
The scalars~$u_j=u_j(t)$ stand for the coordinates of the input vector,~$u\eqqcolon(u_1,\dots,u_{M_\sigma})$,  at time~$t$.

Furthermore, we would like that~$y(t,\Bigcdot)$ converges to~$y_\ttt(t,\Bigcdot)$ exponentially with an arbitrary apriori given rate~$\mu>0$. Finally, we look for an input in feedback form as
  \begin{equation}\notag
u(t)=\bfK(t,y(t,\Bigcdot)-y_\ttt(t,\Bigcdot))\in\bbR^{M_\sigma},\qquad\mbox{shortly } u=\bfK(y-y_\ttt),
\end{equation}
depending only on time~$t$ and on the difference between the current controlled  state~$y(t,\Bigcdot)$ and the targeted  state~$y_\ttt(t,\Bigcdot)$ at time~$t$.

Let us denote by~$U_M$ the  linearly independent set of actuators, 
\begin{align}
&U_M\coloneqq\{\Phi_i^M\mid 1\le j\le {M_\sigma}\},\qquad\clU_M\coloneqq\linspan U_M,\qquad \dim\clU_M={M_\sigma},\label{UM-intro}
\end{align}
and introduce the control operator
\begin{align}
&U_M^\diamond\colon\bbR^{M_\sigma}\to\clU_M,\qquad U_M^\diamond u\coloneqq\sum_{j=1}^{M_\sigma}u_j\Phi_j^M.\notag
\end{align}

This allows us to write the controlled dynamics as
\begin{subequations}\label{sys-NS-K}
\begin{align}
 &\tfrac{\p}{\p t} y -  \nu\Delta y  +\langle  y\cdot\nabla\rangle  y + \nabla p=f+U_M^\diamond\bfK(y-y_\ttt),\qquad \diver y = 0,\\
 &\clG y\rest{\p\Omega}= 0,\qquad y(0,\Bigcdot)= y_0.
\end{align} 
\end{subequations}

 Considering~\eqref{sys-NS} and~\eqref{sys-NS-K} as evolutionary equations in a suitable Hilbert space,  we can omit the spatial variable and write, for simplicity, $y_\ttt(t)\coloneqq y_\ttt(t,\Bigcdot)$ and $y(t)\coloneqq y(t,\Bigcdot)$.
The difference~$y(t)-y_\ttt(t)$ will  lie in a Hilbert space~$\bfV$. Our task is to find, for a given~$\mu>0$, a feedback operator~$\bfK\colon\bfV\to\bbR^{M_\sigma}$, 
so that
\begin{equation}\label{goal-exp}
\begin{split}
&\norm{y(t)-y_\ttt(t)}{\bfV}\le \rme^{-\mu (t-s)}\norm{y(s)-y_\ttt(s)}{\bfV},\quad\mbox{for all }t\ge s\ge 0\mbox{ and }  y_0-y_{\ttt0}\in\bfV.
\end{split}
\end{equation}

\subsection{The main stabilizability result}\label{sS:intro-main-res}
Let us consider the Hilbert space
\begin{equation}\label{spaceH}
\bfH=\{h\in (L^2(\Omega))^2\mid \diver h = 0\mbox{ and }(h\cdot \bfn)\rest{\p\Omega}=0\}.
\end{equation}
of square integrable divergence free vector fields, which are tangent to the boundary.
Initial states shall be taken in the subspace~$\bfV\coloneqq\bfH\bigcap (W^{1,2}(\Omega))^2$, where~$W^{1,2}(\Omega)$ stands for the usual Sobolev subspace of functions defined in~$\Omega$, which are Lebesgue square integrable and have square integrable  first order partial derivatives.

Let~$\vol(\Omega)$ be the volume (i.e., the area) of the spatial domain~$\Omega$. We can choose the total volume $\overline{\rm vol}\in(0,\vol(\Omega)]$ to be covered by the actuators.
We shall need a large enough number~$M_\sigma$ of actuators to stabilize the system. This number may depend on~$\overline{\rm vol}$, and will increase with~$\mu$ and~$\nu^{-1}$; here~$\nu$ is the viscosity parameter as in~\eqref{sys-NS-K}, and~$\mu$ is as in~\eqref{goal-exp}. For this reason, we consider a family~$U_M$ of actuators for each integer~$M$, which will enable us to increase~$M_\sigma$ by increasing the index~$M$.
Appropriate sequences~$(U_M)_{M\in\bbN_+}$ of families~$U_M$ will be given later on.
The main result of this manuscript is that, for an arbitrary given~$\mu>0$, the goal~\eqref{goal-exp} can be achieved by explicitly given feedback operators as
 \begin{align}
\bfK\colon\bfV\to\bbR^{M_\sigma},\qquad\bfK=\bfK_M^\lambda&\coloneqq - \lambda (U_M^\diamond)^{-1}\bfP_{\clU_{M}}^{\widetilde\clU_{M}^{\perp\bfV}}\bfA \bfP_{\widetilde\clU_{M}}^{\clU_{M}^{\perp\bfV}},\label{FeedKy2}
\end{align}
provided that~$\lambda>0$ and~${M_\sigma}=\dim\clU_M\in\bbN$ are large enough.  Here~$\bfA$ stands for the Stokes operator and~$\bfP_{F}^{G^{\perp\bfV}}$ stands for the oblique projection in~$\bfV$ onto~$F$
along~$G^{\perp\bfV}$, where~$G^{\perp\bfV}$ stands for the orthogonal complement to~$G$ in~$\bfV$. Thus, these projections depend on the scalar product in~$\bfV$, which we shall take as~$(y,z)_\bfV\coloneqq(\curl y,\curl z)_{L^2(\Omega)}$, where~$\curl y(x)\in\bbR$ denotes the vorticity function defined as in~\eqref{curl}. The space~$\widetilde\clU_{M}$ is spanned by appropriate auxiliary vector functions, which are used/needed due to some regularity issues, that is, besides the family of actuators in~\eqref{UM-intro}, we consider also
\begin{align}\notag
&\widetilde U_M\coloneqq\{\widetilde \Phi_i^M\mid 1\le j\le {M_\sigma}\},\qquad\widetilde \clU_M\coloneqq\linspan \widetilde U_M,\qquad \dim\widetilde \clU_M={M_\sigma},
\end{align}
where the elements of~$\widetilde U_M$ will be ``more regular'' than those in~$U_M$. These families will be  explicitly constructed later on, in Section~\ref{sS:act}, depending on~$\Omega$ and~$\overline{\rm vol}$ only.
In particular, we will have~$\clU_M+\widetilde\clU_M^{\perp\bfV}=\bfV=\widetilde\clU_M+\clU_M^{\perp\bfV}$ with~$\clU_M\bigcap\widetilde\clU_M^{\perp\bfV}=\{0\}=\widetilde\clU_M\bigcap\clU_M^{\perp\bfV}$, which allow us to define the oblique projections in~\eqref{FeedKy2}.

Througout this manuscript we shall fix~$\overline{\rm vol}\in(0,\vol(\Omega)]$ and~$\nu>0$.  The main result  will follow under a boundedness assumption for the vorticity.
 \begin{assumption}\label{A:bddwtar}
The vorticity~$w_\ttt=\curl y_\ttt$ of the targeted vector field trajectory~$y_\ttt$
satysfies $\norm{w_\ttt}{L^\infty((0,+\infty),W^{1,2}(\Omega))}\coloneqq C_{\tt t}<+\infty$.
 \end{assumption}
  Under this assumption, we shall show the following.
  \begin{mainresult}\label{MR:stab}
Given~$\mu>0$, there exist~$M_*=M_*(\mu,C_{\tt t})$ and~$\lambda_*=\lambda_*(\mu,C_{\tt t})$ such that the solution of~\eqref{sys-NS-K} with the feedback~\eqref{FeedKy2} satisfies~\eqref{goal-exp}, for all~$M\ge M_*$ and~$\lambda\ge\lambda_*$.  
Furthermore, the vorticity~$\curl\Phi_j^M$ of each actuator~$\Phi_j^M$ is supported in a subset~$\overline\omega_j^M\subset\Omega$, with~$\vol\left({\textstyle\bigcup_{j=1}^{M^2}}\omega_j^M\right)=\overline{\rm vol}$. Finally, the constants~$M_*$ and~$\lambda_*$ can be chosen independently of each other.
 \end{mainresult}

We follow the strategy used in~\cite{KunRodWal21} for a class of semilinear parabolic-like equations, introduced in~\cite{Rod21-aut} for  linear case. This strategy is appropriate to derive stabilizability results in a pivot space norm ($\bfH$ in our case), provided we show/have suitable continuity/boundedness properties for the operators defining our dynamics. Here we show stabilizability in the stronger norm of~$\bfV$. For this purpose we write the  2D Navier--Stokes equation in vorticity form and show the stabilizability of the vorticity in an appropriate pivot space~$H\coloneqq\curl\bfV$.

We shall write the  2D Navier--Stokes-like equations satisfied by the difference~$z\coloneqq y-y_\ttt$ in vorticity form and show that the resulting scalar parabolic equation satisfies the regularity and boundedness assumptions required  in the abstract setting in~\cite{KunRodWal21}. This task involve the derivation of appropriate estimates, using appropriate Sobolev embeddings, Young inequalities, Agmon inequalities, and interpolation results. The results in~\cite{KunRodWal21} by themselves would lead to a semiglobal stabilizability result~\cite[Thm.~3.1]{KunRodWal21} where~$(M_*,\lambda_*)$ depends also on an upper bound for the norm~$\norm{ y_0-y_{\ttt0}}{\bfV}$ of the initial difference. To derive the global result we shall use a particular property of the (vorticity of the) nonlinear term~$\langle  y\cdot\nabla\rangle  y$.

\subsection{On applications to observer design}\label{sS:intro-main-res-data}
Assume that the state~$y_\ttt$ of~\eqref{sys-NS} is not known and that we want to estimate it using the output of sensor measurements. In real world applications we will likely have at our disposal a finite number of sensors only. Taking this into account, if we look at the vector fields~$\Phi_j=\Phi_j(x)$, $1\le j\le M_\sigma$, as sensors (see~\cite{Rod21-aut}), then we can explore the particular structure of the operator~$\bfK_M^\lambda$ in~\eqref{FeedKy2}, and use the strategy in applications to observer design~\cite{Rod21-jnls}, also known as continuous data assimilation~\cite{AzouaniOlsonTiti14}. Indeed,  recalling~$\bfP_{\widetilde\clU_{M}}^{\clU_{M}^{\perp\bfV}}=\bfP_{\widetilde\clU_{M}}^{\clU_{M}^{\perp\bfV}}\bfP_{\clU_{M}}^{\clU_{M}^{\perp\bfV}}$, we can write
\begin{equation}\notag
\bfK_M^\lambda(y-y_\ttt)=\bfK_M^\lambda \bfP_{\clU_{M}}^{\clU_{M}^{\perp\bfV}}(y-y_\ttt).
 \end{equation}
 Now, if we look at the elements~$\Phi_j^M$ of~$U_M$ as sensors measuring the ``generalized average''~${s_{\rm tar}}_j\coloneqq(\curl y_\ttt,\curl \Phi_j^M)_{L^2(\Omega)}=(y_\ttt,\Phi_j^M)_{\bfV}$ of the vorticity, giving us the output vector~$Zy_\ttt \coloneqq {s_{\rm tar}}\in\bbR^{M_\sigma}$, we can define the output injection operator
\begin{equation}\notag
\bfJ\colon \bbR^{M_\sigma}\to\bfV,\qquad \bfJ=\bfJ_M^\lambda=U_M^\diamond\bfK_M^\lambda U_M^\diamond\bfV_M^{-1},
 \end{equation}
 where~$\bfV_M\in\bbR^{M_\sigma\times M_\sigma}$ is the matrix with entry~$(\Phi_i^M,\Phi_j^M)_{\bfV}$ in the $i$th row and~$j$th column. In this case we will have the relation
 \begin{equation}\notag
 U_M^\diamond\bfV_M^{-1}Z= \bfP_{\clU_{M}}^{\clU_{M}^{\perp\bfV}},
 \end{equation}
 which implies~$U_M^\diamond\bfK_M^\lambda=\bfJ_M^\lambda Z$.
Therefore, we can see
system~\eqref{sys-NS-K} as a Luenberger observer for system~\eqref{sys-NS},
 \begin{subequations}\label{sys-NS-K-obs}
\begin{align}
 &\tfrac{\p}{\p t} y -  \nu\Delta y  +\langle  y\cdot\nabla\rangle  y + \nabla p=f+\bfJ_M^\lambda (Zy-{s_{\rm tar}}),\qquad \diver y = 0,\\
 &\clG y\rest{\p\Omega}= 0,\qquad y(0,\Bigcdot)= y_0,
\end{align} 
\end{subequations}
where now we see~$y(t)$ as an estimate for~$y_\ttt(t)$.

Then we can interpret/rewrite Main Result~\ref{MR:stab} as follows.
 \begin{maincorollary}\label{MC:obs}
Given~$\mu>0$, there exist~$M_*=M_*(\mu,C_{\tt t})$ and~$\lambda_*=\lambda_*(\mu,C_{\tt t})$ such that the fluid velocity estimate provided by the observer~\eqref{sys-NS-K-obs} satisfies~\eqref{goal-exp}, for all~$\lambda\ge\lambda_*$ and~$M\ge M_*$, where~$M_*$ and~$\lambda_*$ can be chosen independently of each other.
 \end{maincorollary}

\subsection{Further literature}\label{sS:litter}
The stabilization of the Navier--Stokes system to/around a targeted solution~$y_\ttt$, by using a finite number of actuators only, has been investigated in several settings. Probably, the first theoretical results are~\cite{BarbuTri04} for a time-independent~$y_\ttt$ (i.e., an equilibrium), and~\cite{BarRodShi11} for a time-dependent~$y_\ttt$.

The above works consider a Riccati based feedback control, constructed to stabilize the linear Oseen--Stokes system obtained by linearizing the dynamics around the targeted state. Such a feedback is able to stabilize the nonlinear dynamics locally, that is, provided that the initial difference~$y_0-y_{\ttt0}$ is small in a suitable norm. Instead, the result stated in Main Result~\ref{MR:stab} is global; no constraint is imposed on the norm of~$y_0-y_{\ttt0}$ in~\eqref{goal-exp}.

Explicitly given feedbacks may require a number of actuators larger than that required by Riccati based feedbacks. However, they are much cheaper from the computational point of view. After spatial discretization, instead of computing the solution~$\Pi\in\bbR^{n\times n}$ of a nonlinear Riccati matrix equation, we need to compute the oblique projections, which involve the inversion of a relatively smaller matrix~$\clP\in\bbR^{M_\sigma\times M_\sigma}$, with entries as~$\clP_{ij}=(\Phi_i,\widetilde\Phi_j)_\bfV$ in the $i$-th row and~$j$-th column; see~\cite[Lem.~2.8]{KunRod19-cocv}.

 We mention also the strategy in~\cite{AzouaniTiti14} which proposes explicit feedbacks analogue to the ones we construct in here, but with a different proof strategy, see also~\cite{AzouaniOlsonTiti14} in an observer design  (continuous data assimilation, state estimation) setting. This strategy is also based on the ``tuning'' of a pair~$(M,\lambda)$, where~$M$ is (related to) the number of actuators and a positive ``gain'' parameter~$\lambda$. The proof strategy in~\cite{AzouaniTiti14} leads to a result where we choose firstly~$\lambda>0$ and then~$M=M(\lambda)$ (see~\cite[Eq.~(24)]{AzouaniTiti14}), while we follow the strategy in~\cite{KunRodWal21}  leading in general to a result where we firstly choose~$M$ and then~$\lambda=\lambda(M)$. Actually, with slightly different arguments, in the particular setting of 2D Navier--Stokes equations and with the feedback control input~$\bfK_M^\lambda$, as in~\eqref{FeedKy2}, we will be able to show one more important feature for applications, namely, that we can choose~$\lambda\ge\lambda_*$ and~$M\ge M_*$ independently of each other.

Though we do not address here the case of boundary controls, we would like to mention the local stabilization works~\cite{BadTakah11}
for a targeted equibrium~$y_\ttt(t)=y_{\ttt0}$ and~\cite{Rod21-amo} for a time-dependent target~$y_\ttt(t)$. These works use Riccati based boundary feedback controls. Again to avoid the potential expensive computations associated  with such feedbacks, a more explicit locally stabilizing boundary feedback is proposed in~\cite[Thm.~2.3]{Barbu12}, see also~\cite[Thm.~4.1]{BarbuMunt12}.

\subsection{Contents and general notation}\label{sS:notation}
The rest of the paper is organized as follows. 
In Section~\ref{S:setting} we gather some functional spaces which are appropriate to investigate the evolution of the velocity field and its vorticity and we address the contruction of the control actuators.
In Section~\ref{S:vort} we recall the dynamics of the vorticity and reformulate the feedback operator in terms of the vorticity.
The proof of the main stabilizability result is given in Section~\ref{S:proofmain}. We validate our theoretical findings through results of simulations presented in Section~\ref{S:numerics}. Finally, in Section~\ref{S:conclusion} we discuss potential future work, including comments on the 3D Navier--Stokes system and on the shape of the actuators.

\medskip
Concerning the notation, we write~$\bbR$ and~$\bbN$ for the sets of real numbers and nonnegative
integers, respectively, and we define $\bbR_+\coloneqq(0\,+\infty)$,  and~$\mathbb
N_+\coloneqq\mathbb N\setminus\{0\}$.

Given Hilbert spaces~$X$ and~$Y$, if the inclusion
$X\subseteq Y$ is continuous, we write $X\xhookrightarrow{} Y$. We write
$X\xhookrightarrow{\rm d} Y$, respectively $X\xhookrightarrow{\rm c} Y$, if the inclusion is also dense, respectively compact.
The space of continuous linear mappings from~$X$ into~$Y$ is denoted by~$\clL(X,Y)$. In case~$X=Y$ we 
write~$\clL(X)\coloneqq\clL(X,X)$.
The continuous dual of~$X$ is denoted~$X'\coloneqq\clL(X,\bbR)$.

The space of continuous functions from a subset~$S\subseteq X$ into~$Y$ is denoted by~$\clC(S,Y)$. The space of
increasing functions, defined in~$\overline{\bbR_+}=[0,+\infty)$ and vanishing at~$0$ is denoted:
\begin{equation}\notag
 \clC_{0,\rm i}(\overline{\bbR_+},\bbR) \coloneqq \{\mathfrak n\!\mid \mathfrak n\in \clC(\overline{\bbR_+},\bbR),
 \quad\!\! \mathfrak n(0)=0,\quad\!\!\mbox{and}\quad\!\!
 \mathfrak n(\varkappa_2)\ge\mathfrak n(\varkappa_1)\;\mbox{ if }\; \varkappa_2\ge \varkappa_1\ge0\}.
\end{equation}
Next,  we denote by~$\clC_{\rm b, i}(X, Y)$ the vector subspace 
\begin{equation}\notag
 \clC_{\rm b, i}(X, Y)\coloneqq
 \left\{f\in \clC(X,Y) \mid \exists\mathfrak n\in \clC_{0,\rm i}(\overline{\bbR_+},\bbR)\;\forall x\in X:\;
\norm{f(x)}{Y}\le \mathfrak n (\norm{x}{X})
\right\}.
\end{equation} 

The scalar product on a Hilbert space~$\clH$  is denoted~$(\Bigcdot,\Bigcdot)_\clH$. Given closed subspaces~$\clF$ and~$\clG$ of~$\clH$, in case 
$\clF\cap \clG=\{0\}$ we say that $\clF+\clG$ is a direct sum and we write $\clF\oplus \clG$ instead.
For a subset~$S\subseteq\clH$, its orthogonal complement is
denoted~$S^{\perp \clH}\coloneqq\{h\in \clH\mid (h,s)_\clH=0\mbox{ for all }s\in S\}$.

Finally, $C,\,C_i$, $i\in\bbN$, stand for unessential positive constants, which may take different values at different places within the manuscript.

\section{Functional setting and families of actuators}\label{S:setting}
We recall the appropriate subspaces of solenoidal vector fields, defined in the bounded convex polygonal domain~$\Omega$,  involved in the analysis of the Navier--Stokes equations, as well as the Stokes operator and the one-to-one correspondence between solenoidal vector fields and their scalar vorticity. Hence, we recall also the appropriate functional setting to deal with the vorticity equation. Finally, we present a strategy to construct explicitly appropriate families~$U_M$ of actuators and~$\widetilde U_M$ of auxiliary vector functions.   

\subsection{Vector and scalar function spaces}\label{sS:fun-sett}
The Lebesgue space~$H\coloneqq L^2(\Omega)$ is considered as pivot space, $H=H'$. This implies that
$H\times H=H'\times H'$. We shall endow the space~$\bfH\subset H\times H$, in~\eqref{spaceH}, with the scalar product inherited from~$H\times H$, which will imply that~$\bfH=\bfH'$.
Next, by considering the subspace~$\bfV\subset\bfH$ as
\begin{equation}\notag
\bfV=\bfH{\,\textstyle\bigcap\,} (W^{1,2}(\Omega))^2,
\end{equation}
we define the Stokes operator as
 \begin{equation}\notag
\bfA\colon\bfV\to\bfV', \qquad \langle \bfA y,v\rangle_{\bfV',\bfV}\coloneqq(\curl y,\curl v)_{L^2(\Omega)},
\end{equation}
which is a bijection between~$\bfV$ and~$\bfV'$,  with domain
 \begin{equation}\notag
\rmD(\bfA) \coloneqq\{h\in\bfH\mid \bfA h\in\bfH\}=\{h\in \bfH{\,\textstyle\bigcap\,} (W^{2,2}(\Omega))^2\mid (\curl h)\rest{\p\Omega}=0\}.
\end{equation}

We shall pay particular attention to the dynamics of the vorticity function~$w\coloneqq\curl y$, which will satisfy  a semilinear parabolic equation, thus we introduce also the spaces associated
to the Dirichlet Laplacian~$A\colon V\to V'$, 
 \begin{equation}\label{A}
 \langle A w,z\rangle_{V',V}\coloneqq(\nabla w,\nabla z)_{L^2(\Omega)^2},\qquad V\coloneqq W^{1,2}_0(\Omega) \coloneqq\{g\in W^{1,2}_0(\Omega)\mid g\rest\Omega=0\},
\end{equation}
with domain
 \begin{equation}\notag
\rmD(A) \coloneqq\{g\in H\mid A g\in H\}= V{\,\textstyle\bigcap\,} W^{2,2}(\Omega) .
\end{equation}
Hereafter, the spaces above are assumed endowed with the scalar products as follows,
 \begin{subequations}\notag
 \begin{align}
 &(w,z)_H\coloneqq(w,z)_{L^2(\Omega)},&&\quad (w,z)\in H\times H;\\
&(w,z)_V\coloneqq\langle A w,z\rangle_{V',V},&&\quad (w,z)\in V\times V;\\
&(w,z)_{\rmD(A)}\coloneqq(A w, A z)_H&&\quad (w,z)\in \rmD(A)\times\rmD(A);\\
&(y,v)_\bfH\coloneqq(y,v)_{L^2(\Omega)^2},&&\quad (y,v)\in \bfH\times\bfH;\\
&(y,v)_\bfV\coloneqq\langle \bfA y,v\rangle_{\bfV',\bfV},&&\quad (y,v)\in \bfV\times\bfV;\\
&(y,v)_{\rmD(\bfA)}\coloneqq(\bfA y, \bfA v)_\bfH&&\quad (y,v)\in \rmD(\bfA)\times\rmD(\bfA).
\end{align}
\end{subequations}

\begin{lemma}\label{L:isocurl-bfVH}
The mapping $\curl\colon\bfV \to H$ is an isometry and a bijection.
\end{lemma}
\begin{proof}
By definition, we find $(y,v)_\bfV=\langle \bfA y,v\rangle_{\bfV',\bfV}=(\curl y,\curl v)_{H}$, which gives us that~$\curl\colon\bfV \to H$ is an isometry. From~\cite[Thm.~1]{AuchAlex02} it follows the surjectivity and injectivity, the latter property is due to the fact that~$\Omega$ is simply connected.
\end{proof}

We define the stream function~$\psi=\psi_y$ associated with the vector field~$y\in\bfV$ as the solution of the Dirichlet Laplace equation~$A \psi\coloneqq\curl y$, that is,
 \begin{equation}\notag
 \psi_y\coloneqq  A^{-1}\curl y.
\end{equation}

Let us consider the adjoint of the vorticity operator,
\begin{equation}\notag
\curl^*\colon H\to \bfV',\qquad \curl^*f\coloneqq (-\tfrac{\p}{\p x_2} f,\tfrac{\p}{\p x_1}f).
\end{equation}
We find that
 \begin{equation}\notag
 A\rest{\rmD(A)}=\curl\curl^* \quad\mbox{and}\quad  \bfA\rest{\rmD(\bfA)}=\curl^*\curl,
\end{equation}
where the former follows from~$(\nabla w,\nabla z)_{H^2}=(\curl^* w,\curl^* z)_{H^2}$ and the latter from the fact that for vector fields~$y,v$ in~$\rmD(\bfA)$, we have
 \begin{equation}\notag
 \langle \bfA w,z\rangle_{\bfV',\bfV}=(\curl w,\curl z)_{H}=(\curl^*\curl w, z)_{\bfH}.
\end{equation}
Hence, we also find that, for~$y\in\rmD(\bfA)$,
 \begin{align}\notag
&  \bfA \curl^* \psi_y=\curl^*\curl \curl^* \psi_y= \curl^*\curl \curl^* A^{-1}\curl y= \curl^*\curl y
  =\bfA y
\end{align}
which implies~$y=\curl^* \psi_y$, hence we can recover~$y$ from its vorticity as
 \begin{equation}\label{y-ot-curly}
y=\curl^* (A^{-1}(\curl y)).
\end{equation}

\begin{lemma}\label{L:curl-DbfVAbfV}
The mapping $\curl\colon\rmD(\bfA)  \to V$ is a bijection.
\end{lemma}
\begin{proof}
The injectivity follows form~$\rmD(\bfA)\subset\bfV$ and Lemma~\ref{L:isocurl-bfVH}. Next, for an arbitrary~$v\in V$, by Lemma~\ref{L:isocurl-bfVH} there exists~$h\in\bfV$ such that~$\curl h=v$, thus~$\curl^*\curl h=\curl^*v\in \bfH$. From~$(\curl^*\curl h, g)_\bfH=(\bfA h, g)_{\bfV',\bfV}$ for all~$g\in \bfV$, we obtain~$\curl^*\curl h=\bfA h\in\bfH$, which implies that~$h\in\rmD(\bfA)$. 
\end{proof}

\begin{remark}
Lemmas~\ref{L:isocurl-bfVH} and~\ref{L:curl-DbfVAbfV} can be derived directly from~\cite[Appendix~I, Prop.~1.4]{Temam01} in the case of smooth domains~$\Omega$.
\end{remark}

Finally, recalling the identity~$-\Delta=\curl^*\curl -\nabla\diver$, we see that, under Lions bounbary conditions, the Stokes operator satisfies, for~$y\in\rmD(\bfA)$ and~$v\in L^2(\Omega)^2$,
 \begin{align}\notag
( \bfA y,v)_{L^2(\Omega)^2}&=(\curl^*\curl y, v)_{L^2(\Omega)^2}\notag\\
  &=(\curl^*\curl y, v)_{L^2(\Omega)^2}-(\nabla\diver y, v)_{L^2(\Omega)^2}
=-(\Delta y, v)_{L^2(\Omega)^2}.\notag
\end{align}
Thus, we have that
 \begin{align}\notag
 \bfA y=-\Delta y,\qquad\mbox{for all }y\in\rmD(\bfA).
\end{align}
Recall also that we have, by \cite[Ch.~1, Thm.~1.4]{Temam01} the orthogonal sum
 \begin{equation}\notag
L^2(\Omega)^2=\bfH\oplus\nabla W^{1,2}(\Omega);\qquad\nabla W^{1,2}(\Omega)=\bfH^{\perp L^2(\Omega)^2}. 
\end{equation}
.

\subsection{The actuators}\label{sS:act}
For a fixed~$M$, the actuators shall be constructed having a vorticity given by a translation of a rescaling of a reference function
\begin{equation}\label{ref-w}
\phi\in L^2(\clO),\qquad\phi(\overline x)>0,\qquad \overline x\in\clO\subset \fkB, 
\end{equation}
where~$\clO\subset\bbR^2$ is an open set with Lipschitz boundary~$\p\clO$  and~$\fkB\coloneqq\{x\in\bbR^2\mid\norm{x}{\bbR^2}<1\}$ stands for the unit Euclidean ball in~$\bbR^2$. For simplicity (without loss of generality), let
\begin{equation}\notag
\frac{\int_\clO(x_1,x_2)\,\rmd\clO}{\int_\clO1\,\rmd\clO}=(0,0),
\end{equation}
that is, $(0,0)\in\bbR^2$ is the center of mass of~$\clO$. In particular, we can take~$\clO=\fkB$.

We denote by~$\indf_{\omega}$ the indicator function of an open subset~$\omega\subseteq\Omega\subset\bbR^2$ ,
\begin{equation}\notag
\indf_{\omega}(x)\coloneqq\begin{cases}
1,&\mbox{ if }x\in\omega,\\
0,&\mbox{ if }x\in\Omega\setminus\overline\omega,
\end{cases}
\end{equation}
and, for a constant~$\overline r>0$ and a vector~$\overline c\in\Omega$, we define the function
\begin{subequations}\label{ref-w-loc}
\begin{align}
& \overline \phi_{\overline r,\overline c}\coloneqq \indf_{\clO_{\overline r,\overline c}}(x)\phi(\overline r^{-1}(x-\overline c)),\qquad
\clO_{\overline r,\overline c}\coloneqq\overline c + \overline r\clO,\\
\intertext{and, for the case~$\clO_{\overline r,\overline c}\subseteq\Omega$, the actuator}
&\Phi\coloneqq\curl^*(A^{-1}\overline \phi_{\tfrac{r}{M}, c_j^M}).
\end{align}
\end{subequations}

In a rectangular domain~$\Omega=(0,L_1)\times(0,L_2)$, we fix
\begin{equation}\notag
r\in(0,\tfrac12\underline L),\qquad \underline L\coloneqq\min\{L_1,L_2\}
\end{equation}
and, for each~$M\in\bbN_+$, we set the family of actuators as
\begin{subequations}\label{UM}
\begin{align}
&U_M\coloneqq\{\Phi_j^M\mid 1\le j\le M_\sigma\},\qquad\Phi_{j}^M=\curl^*(A^{-1}\overline \phi_{\tfrac{r}{M}, c_j^M}),\qquad M_\sigma\coloneqq M^2,\\
&c_j^M\coloneqq c_\bfj^M=\left(\tfrac{(2\bfj_1+1)L_1}{2M},\tfrac{(2\bfj_2+1)L_2}{2M}\right),\quad \bfj\in\{1,\dots,M\}\times\{1,\dots,M\},
\end{align}
where~$j\mapsto \bfj=(\bfj_1,\bfj_2)$ is a bijection from~$\{1,\dots,M^2\}$ onto~$\{1,\dots,M\}\times\{1,\dots,M\}$.
\end{subequations}

We see that~$\curl\Phi_j^M=\overline \phi_{\tfrac{r}{M}, c_j^M}$ is supported in the closure~$\overline\omega_j^M$ of~$\omega_j^M$,
\begin{equation}\label{suppAct}
\supp(\curl\Phi_j^M)=\overline\omega_j^M=c_j^M+\tfrac{r}{M}\overline\clO,
\end{equation}
and that~$c_j^M$ is the center of mass of~$\omega_j^M$.
Note that the supports of the actuators are pairwise disjoint, because for~$x\in \omega_j^M$ and~$z\in \omega_i^M$, $\{i,j\}\subset\{1,\dots,M^2\}$, we have
\begin{equation}\notag
x-z=c_j^M- c_i^M+\tfrac{r}{M}w,\quad\mbox{with}\quad w\in\clO-\clO,
\end{equation}
which implies that~$\norm{c_j^M- c_i^M}{\bbR^2}\le\norm{x-z}{\bbR^2}+2\tfrac{r}{M}$, since~$\clO\subset\fkB$.
Now, if~$i\ne j$, by~\eqref{UM}, we have that~$\norm{c_j^M- c_i^M}{\bbR^2}\ge\tfrac{3}{2M}\underline L$,
which 
gives us~$\norm{x-z}{\bbR^2}\ge\tfrac{1}{2M}(3\underline L-4r)>0$, for all~$(x,y)\in \omega_j^M\times \omega_i^M$, that is,~$\omega_j^M\bigcap \omega_i^M=\emptyset$, if~$j\ne i$.
 Therefore, we find
\begin{equation}\label{total_vol}
\vol\left({\textstyle\bigcup\limits_{j=1}^{M^2}}\omega_j^M\right)=M^2 \vol(\tfrac{r}{M}\clO)=r^2\vol(\clO)
\end{equation}
for the total volume covered by the actuators, which is independent of~$M$.

The choice of the function~$\phi$ in~\eqref{ref-w} is at our disposal. Once we have chosen~$\phi$, the construction above will give us a stabilizing family of actuators, for large enough~$M$. 

By a technical reason we need a set of  auxiliary vector functions~$\widetilde U_M$, which we construct as above, but with a more regular reference function as
\begin{equation}\label{ref-tildew}
\widetilde \phi\in W^{1,2}_0(\clO),\qquad\widetilde\phi(\overline x)>0,\qquad \overline x\in\clO\subset \fkB,
\end{equation}
thus arriving at the analogue of~\eqref{ref-w-loc}, 
\begin{subequations}\label{ref-w-loc-aux}
\begin{align}
& \overline {\widetilde \phi}_{\overline r,\overline c}\coloneqq \indf_{\clO_{\overline r,\overline c}}(x){\widetilde \phi}(\overline r^{-1}(x-\overline c)),\qquad
\clO_{\overline r,\overline c}\coloneqq\overline c + \overline r\clO,\\
\intertext{and, for~$\clO_{\overline r,\overline c}\subseteq\Omega$, the auxiliary function}
&\widetilde\Phi\coloneqq\curl^*(A^{-1}\overline {\widetilde \phi}_{\tfrac{r}{M}, c_j^M})
\end{align}
\end{subequations}

For a rectangular domain we obtain the analogue of~\eqref{UM},
\begin{align}\label{tildeUM}
&\widetilde U_M\coloneqq\{\widetilde \Phi_j^M\mid 1\le j\le M_\sigma\},\qquad\widetilde \Phi_{j}^M=\curl^*(A^{-1}\overline {\widetilde \phi}_{\tfrac{r}{M}, c_j^M}),
\end{align}
again with~$M_\sigma= M^2$ the~$c_j^M$ as in~\eqref{UM}. 

Fig.~\ref{fig.suppActRect} illustrates the location of the actuators in a rectangular domain; see~\cite[Fig.~1]{KunRodWal21}.
The key point is that the configuration for~$M>1$ corresponds to rescaled copies of the configuration for~$M=1$; one of such copies
is highlighted in Fig.~\ref{fig.suppActRect} by a dashed-border rectangle. This is possible because a rectangle can be decomposed into~$M^2$ similar rescaled copies of itself.
A triangle can also be (up to a rotation) decomposed into~$4^{M-1}$ similar rescaled copies of itself, see Fig.~\ref{fig.suppActTri} for an illustration; with a rotated (by 180 degrees) copy of the configuration for~$M=1$ highlighted by a dashed-border triangle. Hence, after a triangulation, the strategy can be applied to polygonal domains.


\setlength{\unitlength}{.0018\textwidth}
\newsavebox{\Rectfw}%
\savebox{\Rectfw}(0,0){%
\linethickness{3pt}
{\color{black}\polygon(0,0)(120,0)(120,80)(0,80)(0,0)}%
}%
\newsavebox{\Rectref}%
\savebox{\Rectref}(0,0){%
{\color{white}\polygon*(0,0)(120,0)(120,80)(0,80)(0,0)}%
{\color{lightgray}\polygon*(45,25)(75,25)(75,55)(45,55)(45,25)}%
}%


\begin{figure}[h!]
\begin{center}
\begin{picture}(500,100)
\put(0,0){\usebox{\Rectfw}}%
\put(0,0){\usebox{\Rectref}}
 \put(190,0){\usebox{\Rectfw}}
 \put(190,0){\scalebox{.5}{\usebox{\Rectref}}}
 \put(250,0){\scalebox{.5}{\usebox{\Rectref}}}
 \put(250,40){\scalebox{.5}{\usebox{\Rectref}}}
 \put(190,40){\scalebox{.5}{\usebox{\Rectref}}}
 \put(380,0){\usebox{\Rectfw}}
 \put(380,0){\scalebox{.3333}{\usebox{\Rectref}}}
\put(420,0){\scalebox{.3333}{\usebox{\Rectref}}}
\put(460,0){\scalebox{.3333}{\usebox{\Rectref}}}
\put(380,26.6666){\scalebox{.3333}{\usebox{\Rectref}}}
\put(420,26.6666){\scalebox{.3333}{\usebox{\Rectref}}}
\put(460,26.6666){\scalebox{.3333}{\usebox{\Rectref}}}
\put(380,53.3333){\scalebox{.3333}{\usebox{\Rectref}}}
\put(420,53.3333){\scalebox{.3333}{\usebox{\Rectref}}}
\put(460,53.3333){\scalebox{.3333}{\usebox{\Rectref}}}
\put(40,85){$M=1$}
\put(230,85){$M=2$}
\put(420,85){$M=3$}

\linethickness{1.5pt}%
{\color{blue}%

\Dashline(250,0)(310,0){4}%
\Dashline(310,0)(310,40){4}%
\Dashline(310,40)(250,40){4}%
\Dashline(250,40)(250,0){4}%

\Dashline(460,0)(500,0){4}%
\Dashline(500,0)(500,26.6666){4}%
\Dashline(500,26.6666)(460,26.6666){4}%
\Dashline(460,26.6666)(460,0){4}%
}
\end{picture}
\end{center}
\caption{Supports~$\overline\omega_j^M$ of the actuators as in~\eqref{suppAct}. $M_\sigma=M^2$.}
\label{fig.suppActRect}
\end{figure}
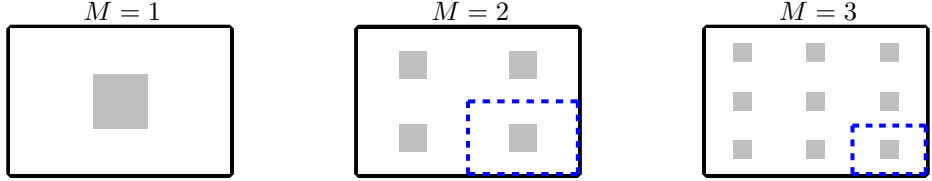


\setlength{\unitlength}{.0018\textwidth}
\newsavebox{\Trifw}%
\savebox{\Trifw}(0,0){%
\linethickness{3pt}
{\color{black}\polygon(0,0)(120,0)(40,80)(0,0)}%
}%

\newsavebox{\Triref}%
\savebox{\Triref}(0,0){%
{\color{white}\polygon*(0,0)(120,0)(40,80)(0,0)}%
{\color{lightgray}\polygon*(30,10)(50,10)(50,30)(30,30)(30,10)}%
}%

\newsavebox{\rTriref}%
\savebox{\rTriref}(0,0){\rotatebox{180}{\usebox{\Triref}}%
}%


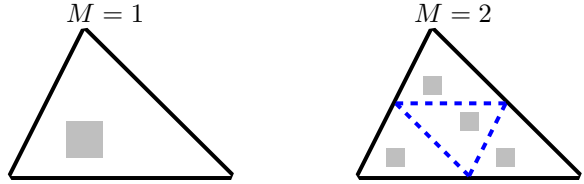
\begin{figure}[h!]
\begin{center}
\begin{picture}(300,100)
\put(0,0){\usebox{\Trifw}}%
\put(0,0){\usebox{\Triref}}
 \put(190,0){\usebox{\Trifw}}
 \put(190,0){\scalebox{.5}{\usebox{\Triref}}}
  \put(250,0){\scalebox{.5}{\usebox{\Triref}}}
 \put(210,40){\scalebox{.5}{\usebox{\Triref}}}
 \put(270,40){\scalebox{.5}{\usebox{\rTriref}}}
\put(30,85){$M=1$}
\put(220,85){$M=2$}

\linethickness{1.5pt}%
{\color{blue}%

\Dashline(270,40)(210,40){4}%
\Dashline(210,40)(250,0){4}%
\Dashline(250,0)(270,40){4}%
}
\end{picture}
\end{center}
\caption{Supports of  actuators for a triangular (sub)domain. $M_\sigma=4^{M-1}$. }
\label{fig.suppActTri}
\end{figure}

\section{The vorticity equation}\label{S:vort}
We write the fluid velocity  stabilizability result, stated at  the end of Section~\ref{sS:intro-main-res}, as a vorticity stabilizability result so that the former will follow from the latter.

We explore the identity~\eqref{y-ot-curly} which gives us a one-to-one correspondence between the velocity field~$y$ and its vorticity~$w=w_y\coloneqq\curl y$. From~\eqref{sys-NS-control} and the discussion in Section~\ref{sS:fun-sett}, we find 
\begin{align}\notag
 &\tfrac{\p}{\p t} w -  \nu\Delta w  +\curl(\langle  y\cdot\nabla\rangle  y) =\curl f+\sum_{j=1}^{M}u_j\curl\Phi_j.
\end{align} 
For the nonlinear term, by direct computations, we obtain
\begin{align}
\curl(\langle  y\cdot\nabla\rangle  y) &=\tfrac{\p}{\p x_2}( y\cdot\nabla  y_1)-\tfrac{\p}{\p x_1}( y\cdot\nabla  y_2)=y\cdot\nabla w+(\tfrac{\p}{\p x_2} y)\cdot\nabla  y_1-(\tfrac{\p}{\p x_1}y)\cdot\nabla  y_2\notag\\
&=y\cdot\nabla w+(\tfrac{\p}{\p x_2} y_1\tfrac{\p}{\p x_1} y_1+\tfrac{\p}{\p x_2} y_2\tfrac{\p}{\p x_2} y_1)-(\tfrac{\p}{\p x_1} y_1\tfrac{\p}{\p x_1} y_2+\tfrac{\p}{\p x_1} y_2\tfrac{\p}{\p x_2} y_2)\notag\\
&=y\cdot\nabla w+w\tfrac{\p}{\p x_1} y_1 +w\tfrac{\p}{\p x_2} y_2=y\cdot\nabla w+w\diver y,\notag
\end{align} 
and since~$\diver y=0$, using~\eqref{y-ot-curly}, we arrive at
\begin{align}\notag
\curl(\langle  y\cdot\nabla\rangle  y) =y\cdot\nabla w=\curl^* (A^{-1}w)\cdot\nabla w
\end{align}
and, with~$w(t)\coloneqq \curl y(t)$ and~$w_0\coloneqq\curl y_0$, we arrive at the vorticity dynamical system
\begin{align}\label{sys-vort-control}
 &\dot w +  \nu A w  +\curl^* (A^{-1}w)\cdot\nabla w =\curl f+\sum_{j=1}^{M}u_j\curl\Phi_j,\qquad w(0)= w_0.
\end{align}

Recall that we are looking for an  input~$u=u(t)$, such that the solution~$y$ of~\eqref{sys-NS-control} converges exponential to the given solution~$y_\ttt$ of~\eqref{sys-NS}. The vorticity~$w_{\ttt}(t)\coloneqq \curl y_\ttt(t)$ of the latter solves, with~$w_{\ttt0}\coloneqq\curl y_{\ttt0}$,
  \begin{align}\label{sys-vort}
 &\dot w_\ttt + \nu A w_\ttt  +\curl^* (A^{-1}w_\ttt)\cdot\nabla w_\ttt =\curl f,\qquad  
w_\ttt(0)=w_{\ttt0},
\end{align}
Again due to the identity~\eqref{y-ot-curly}, we can rewrite our goal in terms of the vorticities, namely, we want that~$w(t)$ converges exponentially to the given target~$w_\ttt(t)$ as~$t\to+\infty$,
\begin{equation}\label{goal-exp-vort}
\begin{split}
&\norm{w(t)-w_\ttt(t)}{H}\le C\rme^{-\mu (t-s)}\norm{w(s)-w_\ttt(s)}{H},\\
&\mbox{for all }t\ge s\ge 0\mbox{ and all }w_0-w_{\ttt0}\in H.
\end{split}
\end{equation}

The vorticity of the actuators and auxiliary functions in~\eqref{UM} and~\eqref{tildeUM}  are
\begin{subequations}\label{Act-Aux-vort}
 \begin{align}
 V_M&\coloneqq\curl U_M=\{\varphi_j^M\mid 1\le j\le M_\sigma\}\subset H,\qquad \varphi_j^M\coloneqq\curl \Phi_j^M=\overline {\phi}_{\tfrac{r}{M}, c_j^M}
\\
\widetilde V_M&\coloneqq\curl \widetilde U_M\coloneqq\{\widetilde\varphi_j^M\mid 1\le j\le M_\sigma\}\subset V,\qquad \widetilde\varphi_j^M\coloneqq\curl  \widetilde\Phi_j^M=\overline {\widetilde \phi}_{\tfrac{r}{M}, c_j^M},
 \end{align}
  \end{subequations}
  with~$M_\sigma$ actuators with vorticities~\eqref{ref-w-loc}, located as in Figs.~\ref{fig.suppActRect} and~\ref{fig.suppActTri}, with the reference functions~${\phi}\in L^2(\clO)$ as in~\eqref{ref-w}, and with auxiliary functions with vorticities as in~\eqref{ref-w-loc-aux} with the reference functions~${\widetilde \phi}\in W^{1,2}_0(\clO)$ as in~\eqref{ref-tildew}.

\subsection{Feedback operator in terms of the vorticity}
We show that in terms of the vorticity, the operator~\eqref{FeedKy2} can be written using oblique projections again, namely, as
  \begin{align}\label{FeedKy2-vort}
 K\in\clL(H,\bbR^M),\qquad K = K_M^\lambda \coloneqq- \lambda (V_M^\diamond)^{-1}P_{\clV_{M}}^{\widetilde\clV_{M}^{\perp H}}A P_{\widetilde\clV_{M}}^{\clV_{M}^{\perp H}},
\end{align} 
where~$P_F^{G^{\perp H}}$ stand for the oblique projection in~$H$ onto~$F$ along~$G^{\perp H}$.

System~\eqref{sys-vort-control}, with the feedback control input reads
\begin{align}\label{sys-vort-K}
 &\dot w +  \nu A w  +\curl^* (A^{-1}w)\cdot\nabla w =\curl f+V_M^\diamond K_M^\lambda (w-w_\ttt),\qquad w(0)= w_0.
\end{align}

That is, the main goal of this section is to show the following result.
  \begin{theorem}\label{T:K=bfKcurl}
   The feedback operators~\eqref{FeedKy2} and~\eqref{FeedKy2-vort} satisfy
$K_M^\lambda \curl =\bfK_M^\lambda$.
  \end{theorem}

We need to derive auxiliry results on the relations between the projections~$\bfP_\bfF^{\bfG^{\perp\bfV}}$ in the space~$\bfV$ of vector fields, used in~\eqref{FeedKy2}, and the projections~$P_F^{G^{\perp H}}$ in the pivot space~$H$ of scalar vorticity functions, used in~\eqref{FeedKy2-vort}.

  \begin{lemma}\label{L:curl-orth} 
Given a subset~$\clF\subset\bfV$, we have the relation
$
\curl (\clF^{\perp\bfV})=(\curl \clF)^{\perp H}.
$
  \end{lemma}
\begin{proof}
For every~$h\in\clF^{\perp\bfV}$ and every~$f\in\clF$, we have~$(\curl h,\curl f)_H=(h, f)_\bfV=0$, which gives us~$\curl (\clF^{\perp\bfV})\subseteq(\curl \clF)^{\perp H}$. Next, for every~$g\in(\curl \clF)^{\perp H}$ and every~$r\in\clF$, we have that~$(\curl^* A^{-1} g, r)_\bfV=(g,\curl r)_H=0$, which gives us~$\psi\coloneqq\curl^* A^{-1} g\in\clF^{\perp\bfV}$ and
$g=\curl\psi$. Hence,~$\curl (\clF^{\perp\bfV})\supseteq (\curl \clF)^{\perp H}$.
  \end{proof}
   \begin{lemma}\label{L:bfP-P}
   Let~$\clF$ and~$\clG$ be closed subspaces of~$\bfV$ such that~$\bfV=\clF\oplus \clG$. Then
   \begin{equation}\notag
   \curl\bfP_{\clF}^{\clG}=P_{\curl\clF}^{\curl\clG}\curl.
   \end{equation}
   \end{lemma}
 \begin{proof}
 Let~$z\in\bfV$ be arbitrary, which we decompose into oblique components as
 \begin{equation}\notag
 z=f+g,\qquad f\coloneqq\bfP_{\clF}^{\clG}z,\quad g\coloneqq\bfP_{\clG}^{\clF}z.
  \end{equation}
  We find that
   \begin{equation}\notag
\curl z=\curl f+\curl g,\qquad \curl f\in\curl  \clF,\quad\curl g\in\curl \clG.
  \end{equation}
Note that~$H=\curl \bfV=\curl(\clF\oplus \clG)=\curl\clF+ \curl\clG$. Finally, note that if~$v\in\curl  \clF\bigcap \curl  \clG$, then there exists~$(f_0,g_0)\in\clF\times\clG$ such that~$v=\curl f_0=\curl g_0$ and we find~$\bfA(f_0-g_0)=\curl^*\curl  (f_0-g_0)=0$ and the injectivity of~$\bfA$ implies that~$f_0=g_0\in\clF\bigcap \clG$. Thus,
  necessarily~$f_0=0$ and~$v=\curl f_0=0$. That is, ~$\curl\clF\bigcap \curl\clG=\{0\}$. We can conclude that
   \begin{equation}\notag
\curl \bfP_{\clF}^{\clG}z=\curl f=P_{\curl\clF}^{\curl\clG}\curl z,\quad\mbox{and}\quad
\curl \bfP_{\clG}^{\clF}z=\curl g=P_{\curl\clG}^{\curl\clF}\curl z,
  \end{equation}
which finishes the proof.
  \end{proof}

 \begin{proof}[Proof of Theorem~\ref{T:K=bfKcurl}]
 For an arbitrary~$z\in\bfV$, we find
 \begin{align}
 -\bfK_M^\lambda z&=  \lambda (U_M^\diamond)^{-1}\bfP_{\clU_{M}}^{\widetilde\clU_{M}^\perp}\bfA \bfP_{\widetilde\clU_{M}}^{\clU_{M}^\perp}z=  \lambda (U_M^\diamond)^{-1}\bfP_{\clU_{M}}^{\widetilde\clU_{M}^\perp}\curl^*\curl \bfP_{\widetilde\clU_{M}}^{\clU_{M}^\perp}z\notag\\
 &=  \lambda (V_M^\diamond)^{-1}\curl\bfP_{\clU_{M}}^{\widetilde\clU_{M}^\perp}\curl^*\curl \bfP_{\widetilde\clU_{M}}^{\clU_{M}^\perp}z\notag
 \end{align}
and, by Lemmas~\ref{L:curl-orth} and~\ref{L:bfP-P} ,
  \begin{align}
- \bfK_M^\lambda z 
 &=  \lambda (V_M^\diamond)^{-1} P_{\clV_{M}}^{\widetilde\clV_{M}^\perp}\curl\curl^* P_{\widetilde\clV_{M}}^{\clV_{M}^\perp}\curl z=  \lambda (V_M^\diamond)^{-1} P_{\clV_{M}}^{\widetilde\clV_{M}^\perp} A P_{\widetilde\clV_{M}}^{\clV_{M}^\perp}\curl z,\notag
 \end{align}
which gives us~$\bfK_M^\lambda z=K_M^\lambda\curl z$, recalling~\eqref{FeedKy2-vort}.
  \end{proof}

\subsection{Main result in vorticity formulation}
 In terms of the vorticity, the main result of this manuscript reads as follows.
 \begin{theorem}\label{T:main-vort}
 Let Assumption~\ref{A:bddwtar} hold true. Given~$\mu>0$, there exist~$M_*=M_*(\mu,C_\ttt)\in\bbN_+$ and~$\lambda_*=\lambda_*(\mu,C_\ttt)>0$ such  that the solutions of~\eqref{sys-vort} and~\eqref{sys-vort-K} satisfy~\eqref{goal-exp-vort}, for every~$M\ge M_*$ and~$\lambda\ge\lambda_*$.  Furthermore, $M_*$ and~$\lambda_*$ can be chosen independently.
 \end{theorem}

\subsection{Proof of main result in vector field formulation}
We show that 
Main Result~\ref{MR:stab} follows as a corollary of Theorem~\ref{T:main-vort}. Indeed, by~\eqref{goal-exp-vort}, we find
\begin{align}
\norm{y(t)-y_\ttt(t)}{\bfV}^2&=\norm{\curl (y(t)- y_\ttt(t))}{H}^2=\norm{w(t)- w_\ttt(t)}{H}^2\notag\\
&\le C\rme^{-\mu t}\norm{w_0- w_{\ttt0}}{H}^2= C\rme^{-\mu t}\norm{\curl (y_0- y_{\ttt0})}{H}^2=C\rme^{-\mu t}\norm{y_0-y_{\ttt0}}{\bfV}^2,\notag
\end{align}
which gives us~\eqref{goal-exp}. Recal also that, by construction, the total volume covered by the actuators is~$\overline\vol\coloneqq r^2\vol(\clO)$; see~\eqref{total_vol}.\qed

\section{Proof of the main result in vorticity formulation}\label{S:proofmain}
This section is dedicated to the proof of Theorem~\ref{T:main-vort}. 

\subsection{Dynamics of the different to the targeted trajectory}
We want the difference~$w(t)-w_\ttt(t)$ between the controlled solution~$w$ solving system~\eqref{sys-vort-control} and the targeted solution~$w_\ttt$ solving system~\eqref{sys-vort} to satisfy~\eqref{goal-exp-vort}. It will be convenient to rescale time, by taking~$\tau\coloneqq\nu t$ and writing
\begin{equation}\notag
z(\tau)\coloneqq w(\nu^{-1}\tau)-w_\ttt(\nu^{-1}\tau) =w(t)-w_\ttt(t).
\end{equation}
Note that, $z(0)=w(0)-w_\ttt(0)= w_0-w_{\ttt0}\eqqcolon z_0$ and that with~$r=\nu s$,
\begin{align}
&\norm{w(t)-w_\ttt(t)}{H}\le C\rme^{-\mu (t-s)}\norm{w(s)-w_\ttt(s)}{H}\notag\\
\Longleftrightarrow\quad&
\norm{z(\tau)}{H}\le C\rme^{-\mu \nu^{-1} (\tau-r)}\norm{z(r)}{H},\notag
\end{align}
That is, ~\eqref{goal-exp-vort} hold true if, and only if, with~$\underline\mu\coloneqq\mu \nu^{-1}$, it holds that
\begin{equation}\label{goal-diff}
\norm{z(\tau)}{H}\le C\rme^{-\underline\mu (\tau-r)}\norm{z(r)}{H},\quad\mbox{for all }t\ge r\ge 0\mbox{ and all }z_0\in H.
\end{equation}

The difference~$z$, with the feedback control input operator~$K_M^\lambda$, satisfies 
\begin{align}\notag
 &\tfrac{\rmd}{\rmd \tau} z +  A z  +\nu^{-1}\curl^* (A^{-1}\underline w)\cdot\nabla\underline  w -\nu^{-1}\curl^* (A^{-1}\underline w_\ttt )\cdot\nabla\underline w_\ttt =\nu^{-1}V_M^\diamond K_M^\lambda z;\\
&\mbox{where}\quad\underline w(\tau)\coloneqq w(\nu^{-1}\tau)\mbox{ and } \underline w_\ttt \coloneqq w_\ttt(\nu^{-1}\tau).\notag
\end{align} 

For the nonlinear terms we find
\begin{align}\notag
&\curl^* (A^{-1}{\underline w})\cdot\nabla{\underline w} -\curl^* (A^{-1}\underline w_\ttt )\cdot\nabla\underline w_\ttt =
  \curl^* (A^{-1}{\underline w})\cdot\nabla z +\curl^* (A^{-1}z)\cdot\nabla\underline w_\ttt \notag\\
 &\qquad= \curl^* (A^{-1}z)\cdot\nabla z+\curl^* (A^{-1}\underline w_\ttt )\cdot\nabla z +\curl^* (A^{-1}z)\cdot\nabla\underline w_\ttt . \notag
\end{align} 
Therefore, the dynamics of~$z={\underline w}-\underline w_\ttt $ satisfies, with~$\underline\lambda\coloneqq \nu^{-1}\lambda$,
\begin{subequations}\label{sys-diff}
\begin{align}
 &\tfrac{\rmd}{\rmd \tau} z +   A z  +A_{\rm rc}z+N(z)=V_M^\diamond K_M^{\underline \lambda} z,\qquad z(0)= z_0,
\intertext{with a reaction-convection operator~$A_{\rm rc}=A_{\rm rc}(\tau)$ and a nonlinear operator~$N$ as}
&A_{\rm rc}z\coloneqq \nu^{-1}\curl^* (A^{-1}z)\cdot\nabla\underline w_\ttt +\nu^{-1}\curl^* (A^{-1}\underline w_\ttt )\cdot\nabla z,\\
&N(z)\coloneqq \nu^{-1}\curl^* (A^{-1}z)\cdot\nabla z.
\end{align} 
\end{subequations}

\subsection{Continuity of the state operators}
We will show that the state operators~$A$, $A_{\rm rc}$, and~$N$,  satisfy the assumptions
as in~\cite[Assumps.~2.1--2.4]{KunRodWal21}, which we write here as the following Lemmas~\ref{LA:A0sp}--\ref{LA:NN}.

\begin{lemma}\label{LA:A0sp}
 $A\in\clL(V,V')$ is an isomorphism from~$V$ onto~$V'$, $A$ is symmetric, and $(y,z)\mapsto\langle Ay,z\rangle_{V',V}$
 is a complete scalar product on~$V.$
\end{lemma}
\begin{proof}
It is well known that the Dirichlet Laplacian~$A$ maps~$V=W^{1,2}_0(\Omega)$ onto~$V'=W^{-1,2}(\Omega)$. From~\eqref{A} it also follows the symmetry. We have that $\langle Ay,z\rangle_{V',V}$ defines a scalar product because, $\langle Ay,y\rangle_{V',V}=0$ implies~$\nabla y=0$, thus~$y=0$, since~$y\in V$.
\end{proof}

\begin{lemma}\label{LA:A0cdc}
The inclusion $V\subseteq H$ is continuous, dense, and compact. 
\end{lemma}
\begin{proof}
The continuity is due to the Poincar\'e inequality. For the compactness of the embedding, recall~\cite[Thm.~4.54]{DemengelDem12}
\end{proof}

\begin{lemma}\label{LA:A1}
Let Assumption~\ref{A:bddwtar} hold true. Then, $A_{\rm rc}(\tau)\in\clL(H,V')$ for (almost) all~$\tau>0$, and  there exists a constant~$D_1\ge0$
such that $\norm{A_{\rm rc}}{L^\infty(\bbR_+,\clL(H,V'))}\le \nu^{-1}D_1C_\ttt$, where~$C_\ttt$ is as in Assumption~\ref{A:bddwtar} and~$D_1$ depends on~$\Omega$ only.
\end{lemma}
\begin{proof}
Let~$v\in V$ be arbitrary. Recalling~\eqref{sys-diff}, we write
\begin{align}
&A_{\rm rc}=\nu^{-1}A_{\rm rc1}+\nu^{-1}A_{\rm rc2},\quad\mbox{with}\notag\\
&A_{\rm rc1}z\coloneqq\curl^* (A^{-1}z)\cdot\nabla\underline w_\ttt 
\quad\mbox{and}\quad
A_{\rm rc2}z\coloneqq\curl^* (A^{-1}\underline w_\ttt )\cdot\nabla z.\notag
\end{align}

For the operator~$A_{\rm rc1}$, we find
\begin{align}
&\langle A_{\rm rc1}z,v\rangle_{V',V}=(\underline w_\ttt \curl^*(A^{-1}z),\nabla v)_H\le C_1\norm{\underline w_\ttt }{L^4}\norm{\curl^*(A^{-1}z)}{(L^4)^2}\norm{\nabla v}{(L^2)^2}\notag
\end{align}
where we have denoted, for simplicity, the Lebesgue spaces~$L^p\coloneqq L^p(\Omega)$.
Let us denote the Sobolev spaces~$W^{s,p}\coloneqq W^{s,p}(\Omega)$.
From the Sobolev embedding~$W^{1,2}\xhookrightarrow{} L^4 $, we find
\begin{align}\notag
\norm{\curl^*(A^{-1}z)}{(L^4)^2}\le C_2\norm{A^{-1}z}{W^{1,4}}\le C_3\norm{A^{-1}z}{W^{2,2}}
\le C_4\norm{z}{L^2},
\end{align}
which gives us, using also the Poincar\'e inequality,
\begin{align}\label{Arc-est1}
\norm{A_{\rm rc1}z}{V'}\le C_5\norm{\underline w_\ttt }{L^4}\norm{z}{H}\le C_6\norm{\underline w_\ttt }{V}\norm{z}{H}.
\end{align}

Next, for the operator~$A_{\rm rc2}$, we obtain
\begin{align}\notag
&\langle A_{\rm rc2}z,v\rangle_{V',V}=(z\curl^*(A^{-1}\underline w_\ttt ),\nabla v)_H\le C_7\norm{z}{L^2}\norm{\curl^*(A^{-1}\underline w_\ttt )}{(L^\infty)^2}\norm{\nabla v}{(L^2)^2}.
\end{align}
By the Agmon inequality (see~\cite[Lem.~13.2]{Agmon65} \cite[Sect.~1.4]{Temam97}) we find that
\begin{align}
\norm{\curl^*(A^{-1}\underline w_\ttt )}{(L^\infty)^2}&\le C_8\norm{\curl^*(A^{-1}\underline w_\ttt )}{W^{2,2}}\le C_9\norm{\curl^*(A^{-1}\underline w_\ttt )}{\rmD(\bfA)}\notag\\
&= C_9\norm{\bfA\curl^*(A^{-1}\underline w_\ttt )}{\bfH}= C_9\norm{\curl^*\underline w_\ttt }{\bfH}= C_9\norm{\underline w_\ttt }{V}\notag
\end{align}
and,   we arrive at
\begin{align}\notag
\norm{A_{\rm rc2}z}{V'}
\le C_{10}\norm{\underline w_\ttt }{V}\norm{z}{H}.
\end{align}

Therefore, $\norm{A_{\rm rc}z}{V'}
\le \nu^{-1}D_1\norm{\underline w_\ttt }{L^\infty(\bbR_+,V)}\norm{z}{H}$, with~$D_1\coloneqq\max\{C_6,C_{10}\}$ and, by Assumption~\ref{A:bddwtar} it follows~$\norm{A_{\rm rc}z}{V'}
\le\nu^{-1}D_1C_\ttt\norm{z}{H}$.
\end{proof}

\begin{lemma}\label{LA:NN}
 We have~$N\in\clC_{\rm b,i}(V,V')$
and
there exists a constant $C_N\ge 0$ such that
 for  all~$(w_1,w_2)\in V\times V$, we have
\begin{align}
\norm{N(w_2)-N(w_1)}{V'}\le &\quad C_N\norm{w_1}{H}^\frac12\norm{w_1}{V}^\frac12\norm{d}{H}+C_N
\norm{w_2}{H}\norm{d}{H}^\frac12\norm{d}{V}^\frac12.\notag
\end{align}
with~$d\coloneqq w_2-w_1$.
\end{lemma}
\begin{proof}
For any~$w\in V$ we find, following the proof of Lemma~\ref{LA:A1}, we arrive at the analogue of~\eqref{Arc-est1},
\begin{align}
&\nu\norm{N(w)}{V'}=\norm{\curl^* (A^{-1}w)\cdot\nabla w}{V'}\le C_5\norm{w}{L^4}\norm{w}{H}\le C_{11}\norm{w}{V}^2,\notag
\end{align} 
which gives us~$N\in\clC_{\rm b,i}(V,V')$.
Next, recalling~\eqref{sys-diff}, we find that, with~$d\coloneqq w_2-w_1$,
\begin{align}
\nu N(w_2)-\nu N(w_1)&=\curl^* (A^{-1}w_2)\cdot\nabla w_2-\curl^* (A^{-1}w_1)\cdot\nabla w_1\notag\\
&=\curl^* (A^{-1}w_2)\cdot\nabla d+\curl^* (A^{-1}d)\cdot\nabla w_1.\notag
\end{align} 
Using again the analogues of~\eqref{Arc-est1} as
\begin{align}
&\norm{\curl^* (A^{-1}d)\cdot\nabla w_1}{V'}\le C_5\norm{w_1}{L^4}\norm{d}{H}\le  C_{12}\norm{w_1}{H}^{\frac12}\norm{w_1}{V}^{\frac12}\norm{d}{H},\notag\\
&\norm{\curl^* (A^{-1}w_2)\cdot\nabla d}{V'}\le C_5\norm{d}{L^4}\norm{w_2}{H}\le  C_{13}\norm{w_2}{H}\norm{d}{H}^{\frac12}\norm{d}{V}^{\frac12},\notag
\end{align}
we conclude that the result follows by taking~$C_N=\nu^{-1}\max\{C_{12},C_{13}\}$. 
\end{proof}

\begin{remark}\label{R:NN}
We observe that Lemma~\ref{LA:NN} implies
\begin{align}\notag
\norm{N(w_2)-N(w_1)}{V'}\le &\quad C_N\sum_{j=1}^2\left(\norm{w_1}{H}^{\zeta_{1j}}\norm{w_1}{V}^{\zeta_{2j}}+\norm{w_2}{H}^{\zeta_{1j}}\norm{w_2}{V}^{\zeta_{2j}}\right)\norm{d}{H}^{\delta_{1j}}\norm{d}{V}^{\delta_{2j}}
\end{align}
with~$(\zeta_{1j},\zeta_{2j},\delta_{1j},\delta_{2j})\in\left\{(\frac12,\frac12,1,0),(1,0,\frac12,\frac12)\right\}$; in either case the relations~$\zeta_{2j}+\delta_{2j}<1$ and~$\delta_{1j}+\delta_{2j}=1$ hold true, as required in~\cite[Assum.~2.4]{KunRodWal21}.
\end{remark}

\subsection{Monotonicity and continuity of the feedback control operator}

\begin{lemma}\label{LA:poincare}
The linear spans~$\clV_{M}\subset H$ of sets of vorticity actuators and~$\widetilde\clV_{M}\subset V$ of auxiliary functions in~\eqref{Act-Aux-vort} satisfy the following.
\begin{enumerate}
\item $M\mapsto M_\sigma$ is a strictly increasing function and~$H= \clV_{M}\oplus  \widetilde\clV_{M}^{\perp H}$;
\item  the  Poincar\'e-like constant
\begin{align}\label{Poinc_const}
\xi_{M_+}\coloneqq\inf_{\varTheta\in (V\bigcap\clV_{M}^{\perp H})\setminus\{0\}}
\tfrac{\norm{\varTheta}{V}^2}{\norm{\varTheta}{H}^2},
\end{align}
satisfies $\lim\limits_{M\to+\infty}\xi_{M_+}=+\infty$.
\end{enumerate}
\end{lemma}
  \begin{proof}
 Clearly~$M\mapsto M_\sigma=M^2$ is strictly increasing, for~$M\in\bbN_+$.  To show that we have the direct sum~$H= \clV_{M}\oplus  \widetilde\clV_{M}^{\perp H}$ we can use~\cite[Lem.~1.7]{KunRod19-cocv} together with the fact that the matrix~$[(\clV_{M},\widetilde\clV_{M})_H]$ with entry~$(\varphi_i^M,\widetilde\varphi_j^M)_H$ in the $i$-th row and $j$-th column is diagonal, with nonzero diagonal entries~$(\varphi_i^M,\widetilde\varphi_j^M)_H>0$; recall that the product~$\overline\varphi_i\overline{\widetilde\varphi}_j\ge0$ has nonempty support~$\overline \omega_j^M$, see~\eqref{suppAct}. To show that~$\lim\limits_{M\to+\infty}\xi_{M_+}=+\infty$ we can follow the arguments in~\cite[Sect.~5]{Rod21-sicon}, using the fact that in the case of a single actuator we have the Poincar\'e-like inequality~\eqref{suppAct},
 \begin{equation}\notag
 \tfrac{ \norm{\nabla f}{W^{1,2}}^2 }{\norm{f}{L^2}^2}\ge \xi_1, \quad\mbox{for all}\quad f\in W^{1,2}{\textstyle\bigcap}\clV_1^{\perp H},
 \end{equation}
 for a suitable constant~$\xi_1>0$. The existence of such a~$\xi_1$ and from the fact that~$c\indf_\Omega\mapsto\fkN(c\indf_\Omega)\coloneqq \norm{(c\indf_\Omega,\varphi_1^1)_{L^2}}{\bbR}$ is a norm in the subspace of constant functions~$\bbR\indf_\Omega\subset W^{1,2}$ (cf.~\cite[Ch. II, Sect. 1.4]{Temam97}). To see that~$\fkN$ is a norm in~$\bbR\indf_\Omega$, from
 \begin{align}
 \fkN(c\indf_\Omega)&=\norm{c}{\bbR}\norm{(\indf_\Omega,\varphi_1^1)_{L^2}}{\bbR}=
\norm{c}{\bbR} \fkN(\indf_\Omega),\notag
\end{align}
it is sufficient to show that~$\fkN(\indf_\Omega)\ne0$, which follows from
$\fkN(\indf_\Omega)
 =\norm{\int_\Omega\varphi_1^1\,\rmd\Omega}{\bbR}
  =\norm{\int_{\omega_j^M}\varphi_1^1\,\rmd\Omega}{\bbR}>0,
$ because~$\varphi_1^1(x)>0$ for~$x\in\omega_j^M$.
  \end{proof} 

\begin{lemma}\label{LA:Kmonot}
The feedback input operator~$\clK_M^1\coloneqq V_M^\diamond K_M^1$ satisfies
  \begin{align}\label{monotK}
  &\clK_M^1z=\clK_M^1(P_{\clV_{M}}^{\clV_{M}^{\perp H}} z)\in\clV_{M}\quad\mbox{for all}\quad z\in H\notag
  \intertext{and, there exists a constant~$\kappa_M>0$ such that}
  &(\clK_M^1p,p)_H\le-\norm{P_{\widetilde\clV_{M}}^{\clV_{M}^{\perp H}} p}{V}^2\le -\kappa_M\norm{p}{H}^2\quad\mbox{for all}\quad p\in \clV_{M}.\notag
\end{align} 
\end{lemma}
  \begin{proof}
Since~$P_{\widetilde\clV_{M}}^{\clV_{M}^{\perp H}}=P_{\widetilde\clV_{M}}^{\clV_{M}^{\perp H}}P_{\clV_{M}}^{\clV_{M}^{\perp H}}$, we can write
\begin{equation}\notag
\clK_M^1z= V_M^\diamond K_M^1z=-P_{\clV_{M}}^{\widetilde\clV_{M}^{\perp H}}A P_{\widetilde\clV_{M}}^{\clV_{M}^{\perp H}}z=-P_{\clV_{M}}^{\widetilde\clV_{M}^{\perp H}}A P_{\widetilde\clV_{M}}^{\clV_{M}^{\perp H}}P_{\clV_{M}}^{\clV_{M}^{\perp H}} z\in\clV_M.
 \end{equation}
The monotonicity follows by the adjoint relation~$(P_{\clV_{M}}^{\widetilde\clV_{M}^{\perp H}})^*=P_{\widetilde\clV_{M}}^{\clV_{M}^{\perp H}}$, see~\cite[Lem.~3.8]{RodSturm20} \cite[Lem.~3.4]{KunRodWal21}, which allows us to obtain
  \begin{align}\label{KmonotV}
    &(\clK_M^1p,p)_H=- (AP_{\widetilde\clV_{M}}^{\clV_{M}^{\perp H}} p,P_{\widetilde\clV_{M}}^{\clV_{M}^{\perp H}} p)_H=-\norm{P_{\widetilde\clV_{M}}^{\clV_{M}^{\perp H}} p}{V}^2.
\end{align} 
Finally, note that since~$p\in\clV_M$ and~$H=\widetilde\clV_M\oplus\clV_M^{\perp H}$, it follows that $P_{\widetilde\clV_{M}}^{\clV_{M}^{\perp H}}p=0$ if, and only if,~$p=0$. Hence,  ~$p\mapsto\norm{P_{\widetilde\clV_{M}}^{\clV_{M}^{\perp H}} p}{V}$ defines a norm in the finite-dimensional space~$\clV_M$ and, consequently, $\norm{P_{\widetilde\clV_{M}}^{\clV_{M}^{\perp H}} p}{V}\ge\kappa_M\norm{p}{H}$, for some constant~$\kappa_M>0$.
  \end{proof}

\begin{remark}\label{R:Kmonot}
Assumption~2.6 in~\cite{KunRodWal21} follows from Lemma~\ref{LA:Kmonot}, with~$\overline\clK_M \coloneqq  \kappa_M^{-1}\clK_M^1$ satisfying~$(\overline\clK_M p,p)_H\le-\norm{p}{H}$.
\end{remark}

From Lemmas~\ref{LA:A0sp}--\ref{LA:NN} and Lemmas~\ref{LA:poincare}--\ref{LA:Kmonot} (see also Rems.~\ref{R:NN} and~\ref{R:Kmonot}) it follows that we can apply~\cite[Thm.~3.1]{KunRodWal21} to obtain the following semiglobal stabilizability result.
 \begin{theorem}\label{T:main-vort-semi}
 Given~$\mu>0$ and~$R>0$, there exist~$M_*=M_*(\mu,R)\in\bbN_+$ and~$\lambda_*=\lambda_*(\mu,R,M)>0$ such  that the solutions of~\eqref{sys-vort} and~\eqref{sys-vort-K} satisfy~\eqref{goal-exp-vort}, for every~$M\ge M_*$ and~$\lambda\ge\lambda_*$, provided
 $\norm{w_0-w_{\ttt0}}{H}\le R$.  
 \end{theorem}
 However, recall that we are looking for a global result where we want to find~$M_*$ and~$\lambda_*$ independent of~$R$, and also independently of each other. We shall prove such result in Section~\ref{sS:proofT:main-vort}.

\subsection{Proof of Theorem~\ref{T:main-vort}}\label{sS:proofT:main-vort}
We follow a variation of the arguments in~\cite[Sect.~3.1]{KunRodWal21}.
We recall that~$z\coloneqq \underline  w-\underline w_\ttt $ satisfies~\eqref{sys-diff},
 where
 \begin{equation}\notag
 \underline w(\tau)\coloneqq w(\nu^{-1}\tau)\mbox{ and } \underline w_\ttt (\tau)\coloneqq w_\ttt(\nu^{-1}\tau),\quad\mbox{and}\quad \tau=\nu t.
 \end{equation}
 
 Let us be given~$\mu>0$ and, recalling~\eqref{goal-diff}, we define~$\underline\mu=\nu^{-1}\mu$.

 Multiplying the dynamics   in~\eqref{sys-diff} by~$2z$, where~$\overline\lambda=\nu^{-1}\lambda$,
 we find
 \begin{align}
\tfrac{\rmd}{\rmd \tau}\norm{z}{H}^2 &=-2\langle Az+A_{\rm rc}z+N(z),z\rangle_{V',V}
+2\bigl(V_M^\diamond K_M^{\underline \lambda} z,z\bigr)_H.\notag
\end{align}
 Now, the nonlinearity has the particular property that
  \begin{equation}\notag
  \langle N(z),z\rangle_{V',V}=\nu^{-1}(\curl^* (A^{-1}z),\nabla(z^2))_{(L^2)^2}=0,
   \end{equation}
leading us to
 \begin{align}
\tfrac{\rmd}{\rmd \tau}\norm{z}{H}^2 &=-2\norm{z}{V}^2-2\langle A_{\rm rc}z,z\rangle_{V',V}
+2\bigl(V_M^\diamond K_M^{\underline \lambda} z,z\bigr)_H,\notag
\end{align}
where the contribution of the nonlinear term is omitted. This possibility of such omission is the reason why we  are going to be able to choose~$M$ and~$\lambda$ independently of the norm of the initial difference~$z_0\coloneqq w_0-\widehat  w_0$. 

From~$\bigl(V_M^\diamond K_M^{\underline \lambda} z,z\bigr)_H=\bigl(V_M^\diamond K_M^{\underline \lambda}(P_{\clV_M}^{\clV_M^{\perp H}}z),P_{\clV_M}^{\clV_M^{\perp H}}z\bigr)_H$, recalling Lemma~\ref{LA:Kmonot}, we obtain
\begin{align}
\bigl(V_M^\diamond K_M^{\underline \lambda} z,z\bigr)_H\le-\underline\lambda\norm{P_{\widetilde\clV_{M}}^{\clV_{M}^{\perp H}} z}{V}^2.\notag
\end{align}
In this estimate, we have the (strong) $V$-norm  in the right hand side, this is the reason why we are going to be able to choose~$M$ and~$\lambda$ independently of each other. Note that, recalling~\eqref{KmonotV}, the appearance of the~$V$-norm is due to the use of the diffusion operator~$A$ in the construction of~$K_M^{\underline \lambda}$ as in~\eqref{FeedKy2-vort}.

Next, we use Lemma~\ref{LA:A1} and find
\begin{align}
\tfrac{\rmd}{\rmd \tau}\norm{z}{H}^2 &\le-2\norm{z}{V}^2+2 C_{\rm rc}\norm{z}{H}\norm{z}{V}
-2\underline\lambda\norm{P_{\widetilde\clV_{M}}^{\clV_{M}^\perp} z}{V}^2,\quad\mbox{with}\quad C_{\rm rc}\coloneqq\nu^{-1}D_1C_\ttt.\notag
\end{align}
Thus, by the Young inequality we  obtain
\begin{align}
 2\langle A_{\rm rc}z,z\rangle_{V',V}&\le2\gamma_1\norm{z}{V}^2 +2\gamma_1^{-1}C_{\rm rc}^2\norm{z}{H}^2,\quad\mbox{for all}\quad \gamma_1>0,\notag
\end{align}
which leads us to
\begin{align}
\tfrac{\ed}{\ed \tau}\norm{z}{H}^2
&\le-(2-2\gamma_1)\norm{z}{V}^2
+\left(2\gamma_1^{-1}C_{\rm rc}^2\right)\norm{z}{H}^2
-2\underline\lambda \norm{P_{\widetilde \clV_M}^{\clV_M^\perp}z}{V}^2.\notag
\end{align}
Now, we set~$\gamma_1=\frac12$, and obtain
\begin{equation}
\tfrac{\ed}{\ed \tau}\norm{z}{H}^2
\le-\norm{z}{V}^2 -2\underline\lambda \norm{P_{\widetilde \clV_M}^{\clV_M^\perp}z}{V}^2
+4C_{\rm rc}^2\norm{z}{H}^2.\notag
\end{equation}

Next, we write~$z=\varTheta+\vartheta$ with~$\varTheta\coloneqq P_{\clV_M^\perp}^{\widetilde \clV_M}z$ and~$\vartheta= P_{\widetilde \clV_M}^{\clV_M^\perp}z$ and observe that
\begin{equation}
-\norm{z}{V}^2=-\norm{\varTheta}{V}^2-2(\varTheta,\vartheta)_V-\norm{\vartheta}{V}^2\le
-\tfrac12\norm{\varTheta}{V}^2+\norm{\vartheta}{V}^2.\notag
\end{equation}
Recalling the Poincar\'e-like constant in~\eqref{Poinc_const}, we arrive at
\begin{equation}
\tfrac{\ed}{\ed \tau}\norm{z}{H}^2
\le-\tfrac12\xi_{M_+}\norm{\varTheta}{H}^2 -(2\underline\lambda-1) \norm{\vartheta}{V}^2
+4C_{\rm rc}^2\norm{z}{H}^2.\notag
\end{equation}

Due to Lemma~\ref{LA:poincare}, we can choose~$M_*\in\bbN_+$ such that
\begin{align}
\xi_{M_+}&\ge 8\underline\mu+16C_{\rm rc}^2,\quad\mbox{for all}\quad M\ge M_*;\notag
\intertext{and (independently) we can choose~$\underline\lambda_*>0$ such that}
\underline\lambda_*&\ge 2\underline\mu+8C_{\rm rc}^2+\tfrac12.\notag
\end{align}

For these choices we obtain that for all~$M\ge M_*$ and all~$\underline\lambda\ge\lambda_*$, it holds
\begin{equation}
\tfrac{\ed}{\ed \tau}\norm{z}{H}^2
\le-(4\underline\mu+8C_{\rm rc}^2)(\norm{\varTheta}{H}^2+\norm{\vartheta}{H}^2)
+4C_{\rm rc}^2\norm{z}{H}^2\le-2\underline\mu\norm{z}{H}^2.\notag
\end{equation}
From~$\tau=\nu t$, $\underline\mu=\nu^{-1}\mu$, and~$z(\tau)=z(\nu t)=w(t)-w_\ttt(t)$,
we obtain, for all~$t\ge s\ge0$,
\begin{align}
\norm{w(t)-w_\ttt(t)}{H}^2&=\norm{z(\tau)}{H}^2\le\ex^{-2\underline\mu(\tau-\nu s)}\norm{z(\nu s)}{H}^2
=\ex^{-2\mu(t- s)}\norm{w(s)-w_\ttt(s)}{H}^2.\notag
\end{align}
Therefore, the theorem follows with~$M_*$ as above and with~$\lambda_*=\nu\underline\lambda_*$. Note that, since~$\underline\lambda=\nu^{-1}\lambda$, we have that~$\lambda\ge\lambda_*$ if, and only if, $\underline\lambda\ge\underline\lambda_*$. 
\qed

\section{Numerical results}\label{S:numerics}
We validate the theoretical results by presenting the results of simulations showing the stabilizing performance of the feedback control. The results are presented for the scalar vorticity equation. The simulations were run with Matlab.

To simplify the exposition, throughout the previous sections we have considered homogeneous boundary conditions~$w\rest{\p\Omega}=0$ for the vorticity. The case of nonhomogeneous boundary conditions~$w\rest{\p\Omega}=g$, with~$g\ne0$, can be reduced to the case of homogeneous boundary conditions  by a lifting argument, assuming that the boundary data~$g$ is regular enough (cf. Rem.~\ref{R:nonhom-bcs}).
Considering more general~$g\ne0$ makes it easier to construct exact analytical solutions for the sake of numerical comparison tests. We shall include an example with such a comparison hereafter.

We consider the target as the solution~$w_\ttt=w_\ttt(t,x)$ of~\eqref{sys-vort}, 
 \begin{subequations}\notag
 \begin{align}
 &\dot w_\ttt  +   \nu A w_\ttt  +\curl^* (A^{-1}w_\ttt)\cdot\nabla w_\ttt =f, \label{sys-vort-num-dyn}\\
& w_\ttt\rest{\p\Omega}=g,
\qquad w_\ttt(0)=w_{\ttt0},
\end{align}
\end{subequations}
where we may have~$g\ne0$.

Recall that we want to confirm that the controlled solution~$w= w(\tau,x)$ of~\eqref{sys-vort-control},
\begin{subequations}\notag
\begin{align}
 &\dot{w} +\nu A  w  +\curl^* (A^{-1} w)\cdot\nabla  w =f+
 \clK_M^\lambda P_{\clV_{M}}^{\clV_{M}^{\perp H}}(w-w_\ttt),\\
 &w\rest{\p\Omega}=g,\qquad w(0)= w_0,
\end{align}
with the feedback input as in~\eqref{FeedKy2-vort} (recall that~$P_{\widetilde\clV_{M}}^{\clV_{M}^{\perp H}}=P_{\widetilde\clV_{M}}^{\clV_{M}^{\perp H}}P_{\clV_{M}}^{\clV_{M}^{\perp H}}$),
\begin{equation}
\clK_M^\lambda \coloneqq V_M^\diamond K_M^\lambda= - \lambda P_{\clV_{M}}^{\widetilde\clV_{M}^{\perp H}}A P_{\widetilde\clV_{M}}^{\clV_{M}^{\perp H}},
\end{equation}
\end{subequations}
converges exponentially to~$w_\ttt$, as time~$\tau\to+\infty$; see Theorem~\ref{T:main-vort}. 

For the diffusion coefficient we have chosen~$\nu=0.01$. 

The spatial domain is a triangular domain~$\Omega\subset\bbR^2$. For a given positive integer~$M$, the set~$V_M$ of actuators and the set~$\widetilde V_M$ of auxiliary functions located as in Fig.~\ref{fig.suppActTri}, are constructed with the reference functions~$\phi$, in~\eqref{ref-w}, and~$\widetilde\phi$, in~\eqref{ref-tildew}, chosen as
\begin{equation}\notag
\begin{split}
&\phi(\overline x)\coloneqq1\quad\mbox{and}\quad \widetilde\phi(\overline x)\coloneqq\sin(\pi(\overline x_1+\tfrac12))\sin(\pi(\overline x_2+\tfrac12)),\\
\mbox{for}\quad
&\overline x\in\clO\coloneqq(-\tfrac12,\tfrac12)\times(-\tfrac12,\tfrac12).
\end{split}
\end{equation}
In particular, the actuators have square supports, as translations of rescaled copies of~$\clO$. The actuators and auxiliary functions are plotted in Fig.~\ref{fig.meshM12}, for the cases~$M\in\{1,2\}$, where we also show the coarsest spatial triangulations~$\clT_0$ used in the simulations.
\begin{figure}[ht]
\centering
\subfigure[Vorticity actuator. $M_\sigma=1$.\label{fig.mesh-vortM1}]
{\includegraphics[width=0.45\textwidth,height=0.33\textwidth]{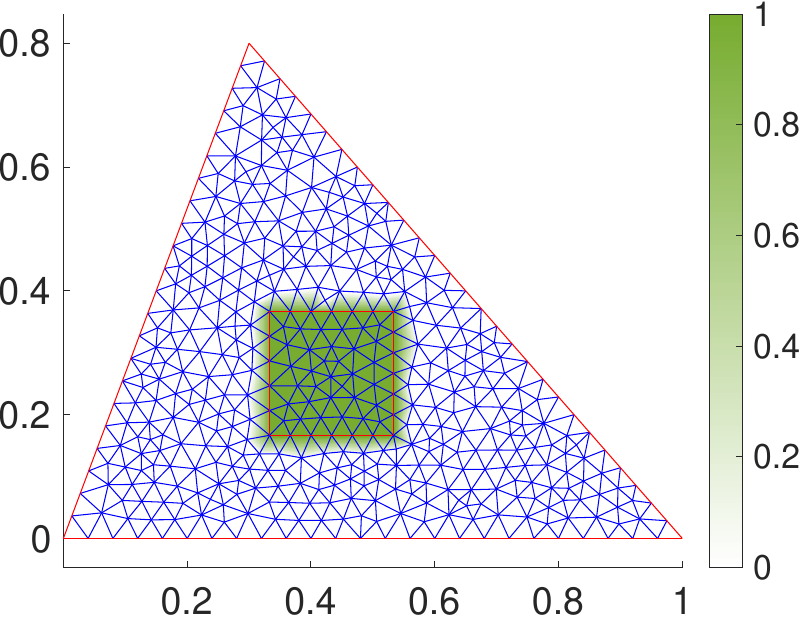}}
\subfigure[Vorticity auxiliary function. $M_\sigma=1$]
{\includegraphics[width=0.45\textwidth,height=0.33\textwidth]{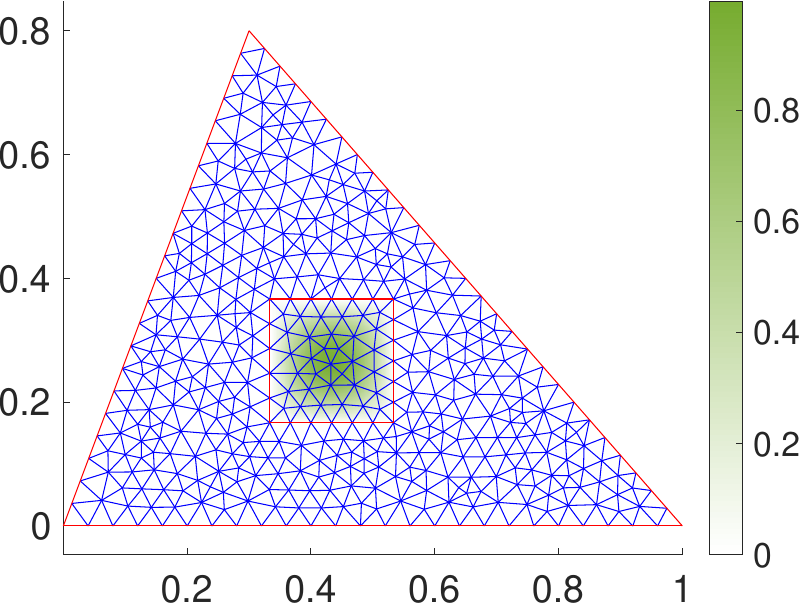}}
\\
\centering
\subfigure[Vorticity actuators. $M_\sigma=4$.\label{fig.mesh-vortM4}]
{\includegraphics[width=0.45\textwidth,height=0.33\textwidth]{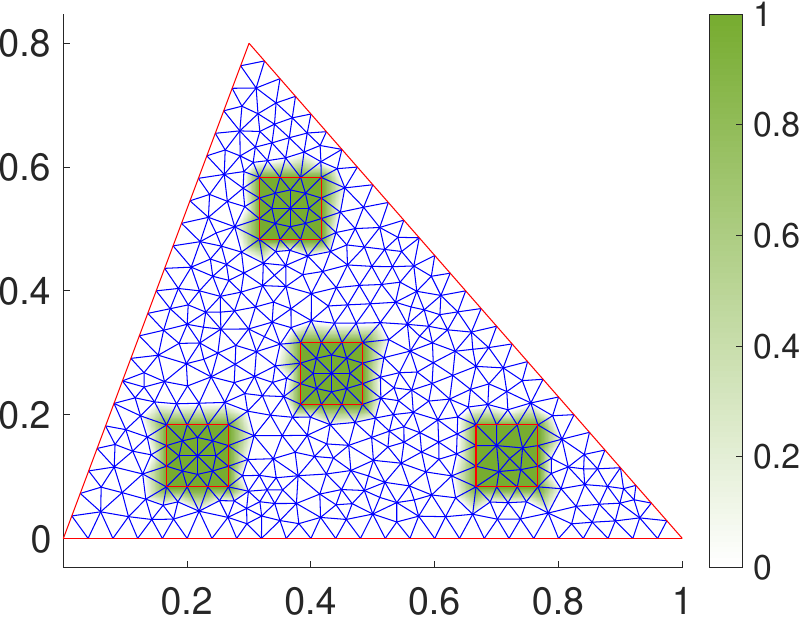}}
\subfigure[Vorticity auxiliary functions. $M_\sigma=4$.]
{\includegraphics[width=0.45\textwidth,height=0.33\textwidth]{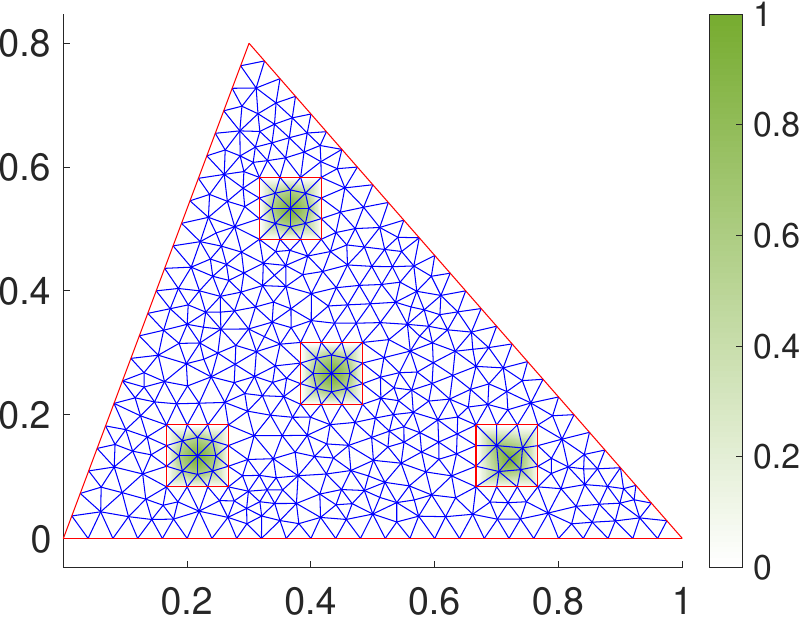}}
\caption{Triangulations~$\clT_0$, $M\in\{1,2\}$, and supports of the actuators.}
\label{fig.meshM12}
\end{figure}

Note that we used an ad-hoc geometry to including the boundary of the supports of the actuators. This is not essential, but provides (at least graphically) a better approximation of the actuators in coarse triangulations (cf.~\cite[Fig.~1]{Rod22-eect}).

We consider regular refinements of these triangulations as
\begin{align}\notag
\clT_\rho,\qquad \rho\in\{0,1,2\},
\end{align}
with~$\clT_{\rho+1}$ obtained from~$\clT_{\rho}$, $\rho\in\{0,1\}$, by dividing each triangle into~$4$ similar triangles. We used the Matlab routines~{\tt initmesh} to generate~$\clT_0$ and~{\tt refinemesh} to generate~$\clT_{\rho+1}$.

Next, we consider the temporal (uniform) discretization of the interval~$[0,+\infty)$ as
\begin{align}\notag
\bft\coloneqq t^{\rm step}\bbN=\{jt^{\rm step}\mid j\in\bbN\},\quad\mbox{with}\quad t^{\rm step}\coloneqq4\cdot10^{-4}
\end{align}
which we refine to arrive at the spatio-temporal discretizations, as
\begin{align}\label{mesh-rho}
\bfR_\rho\coloneqq(\clT_\rho,\bft_\rho)\quad\mbox{with}\quad \bft_\rho= 2^{-\rho}\bft;\qquad\mbox{for}\quad\rho\in\{0,1,2\}.
\end{align}

\begin{example}\label{Ex:unk}
We take the boundary and body external forcings as
\begin{subequations}\label{data:unk}
\begin{align}
&g(t,x)=0,\\
&f(t,x)=2\cos(2t)\sign(x_1-\tfrac3{10}+\tfrac1{10}\cos(4t))\sign(x_2-\tfrac3{10}+\tfrac1{10}\sin(4t)).
\intertext{and the initial states as}
&w_{\ttt0}(x)=0,\qquad w_0(x)=-2\sin(3 x_1)+1.
\end{align}
\end{subequations}

We considered the cases~$M\in\{1,2\}$ and performed the simulations using the spatio-temporal mesh~$\bfR_\rho$, with~$\rho=1$.

In Fig.~\ref{fig.norm-unk} we present the evolution of the free dynamics, corresponding to the case~$\lambda=0$. We see that~$w(t)$ converges exponentially to~$w_\ttt(t)$. However, we recall that we want to check that an arbitrarily large decrease rate~$\mu$ can be achieved; see Theorem~\ref{T:main-vort}. In the same Fig.~\ref{fig.norm-unk} we can indeed see that a larger rate can be achieved by increasing~$M$.  For example, with~$4$ actuators we obtain  an exponential rate as~$\mu\approx\tfrac{35}{20}=1.75$, whereas the stability rate of the free dynamics is approximately~$\mu\approx\tfrac{16}{24}=0.66$. The oscillations that we see for time~$t>20$ and~$M=4$ is due to the fact that we have reached the used standard Matlab machine precision~${\tt eps}\approx2\cdot10^{-16}\approx\rme^{-36}$.
\begin{figure}[ht]
\centering
\subfigure
{\includegraphics[width=0.45\textwidth,height=0.375\textwidth]{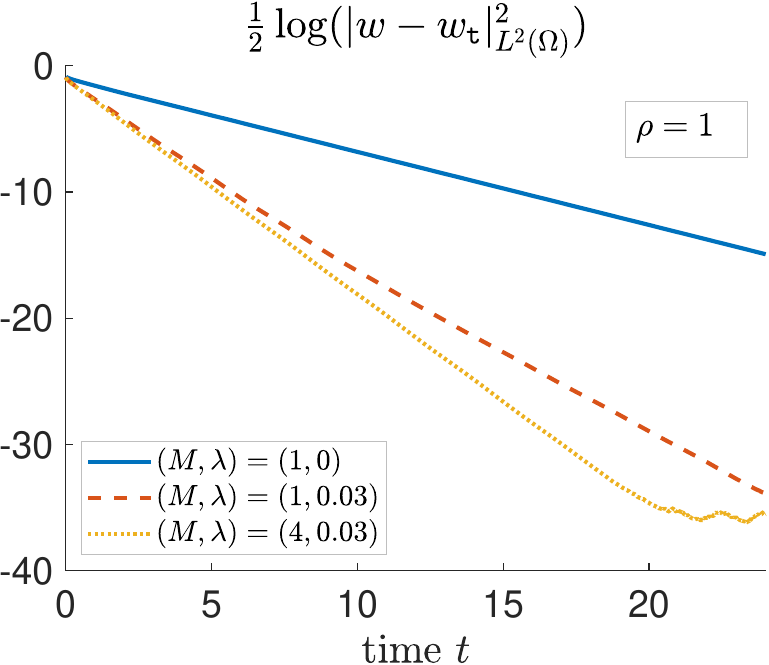}}
\caption{Increasing~$M$. Data~\eqref{data:unk}.}
\label{fig.norm-unk}
\end{figure}

In Fig.~\ref{fig.snap-unkM1-target} we present time-snapshots of the stream function of the targeted state. The associated velocity field is tangent to the plotted streamlines, pointing clockwise  around a local maximum and counterclockwise around a local minimum.

In Fig.~\ref{fig.snap-unkM1-diffBu} we present  time-snapshots of the stream function corresponding to the difference between the controlled~$y$ and the targeted~$y_\ttt$ states, and also  time-snapshots of the stream function of the corresponding control forcing 
\begin{align}\notag
Bu\coloneqq U_M^\diamond K_M^\lambda (w-w_\ttt)\in\clU_M,
\end{align} 
where in particular we can see the shape of the stream function of the actuator, with an extremum at the location of the support of its vorticity in Fig.~\ref{fig.mesh-vortM1}.

\begin{figure}[ht]
\centering
\subfigure
{\includegraphics[width=0.32\textwidth,height=0.25\textwidth]{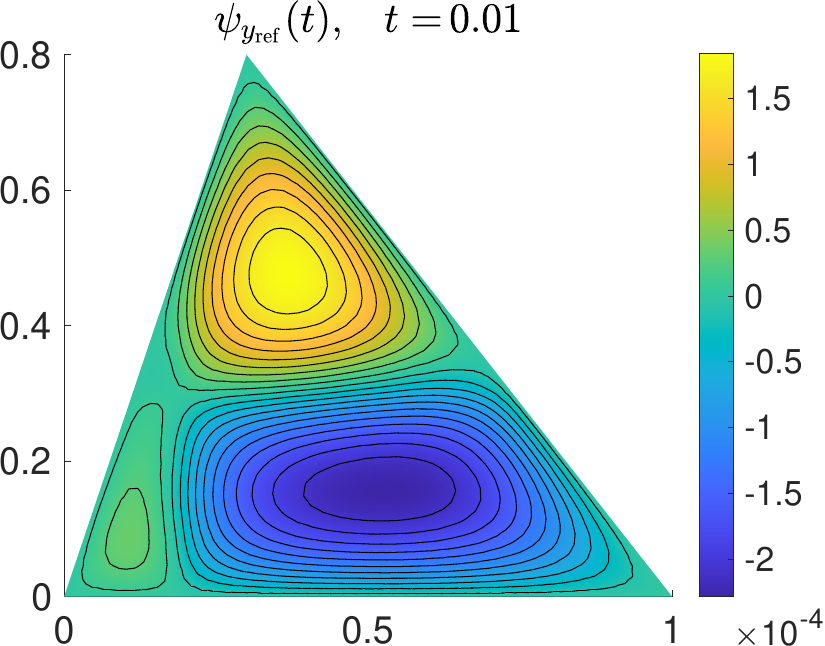}}
\subfigure
{\includegraphics[width=0.32\textwidth,height=0.25\textwidth]{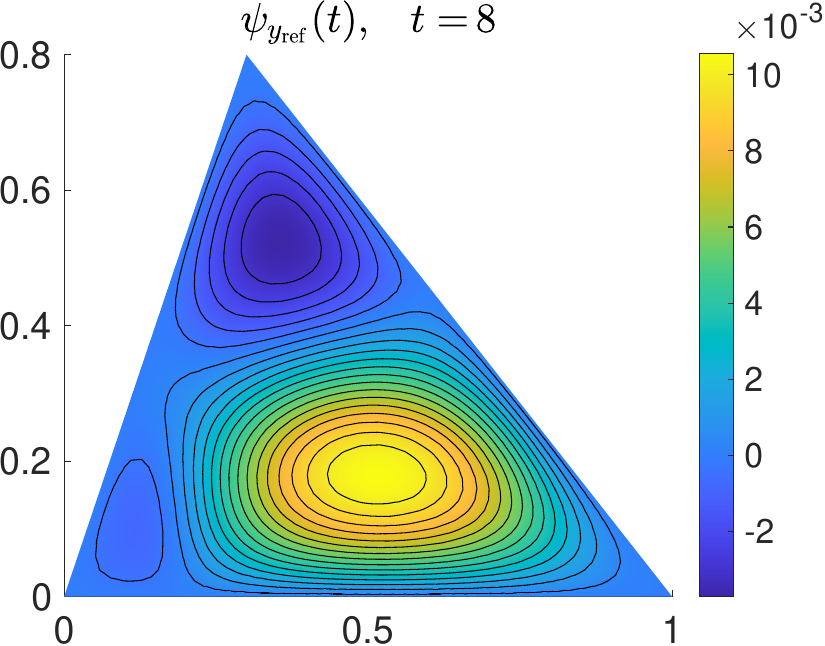}}
\subfigure
{\includegraphics[width=0.32\textwidth,height=0.25\textwidth]{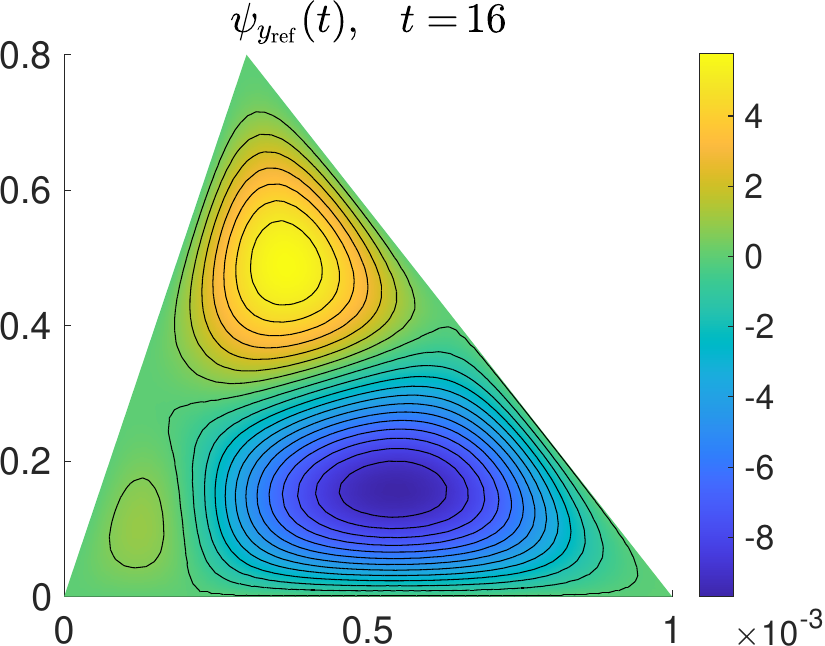}}
\caption{Stream function of the target. Data~\eqref{data:unk}.}
\label{fig.snap-unkM1-target}
\end{figure}

\begin{figure}[ht]
\centering
\subfigure
{\includegraphics[width=0.32\textwidth,height=0.25\textwidth]{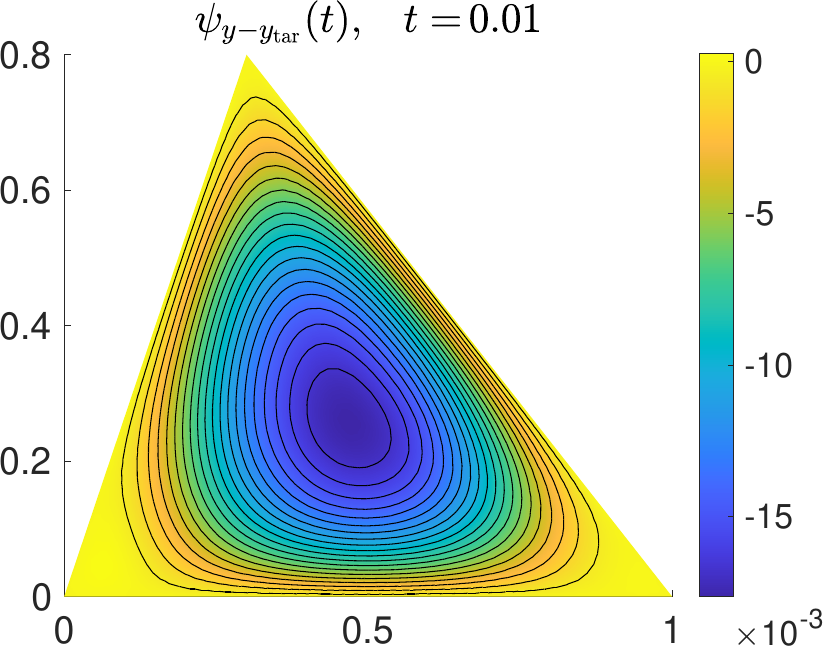}}
\subfigure
{\includegraphics[width=0.32\textwidth,height=0.25\textwidth]{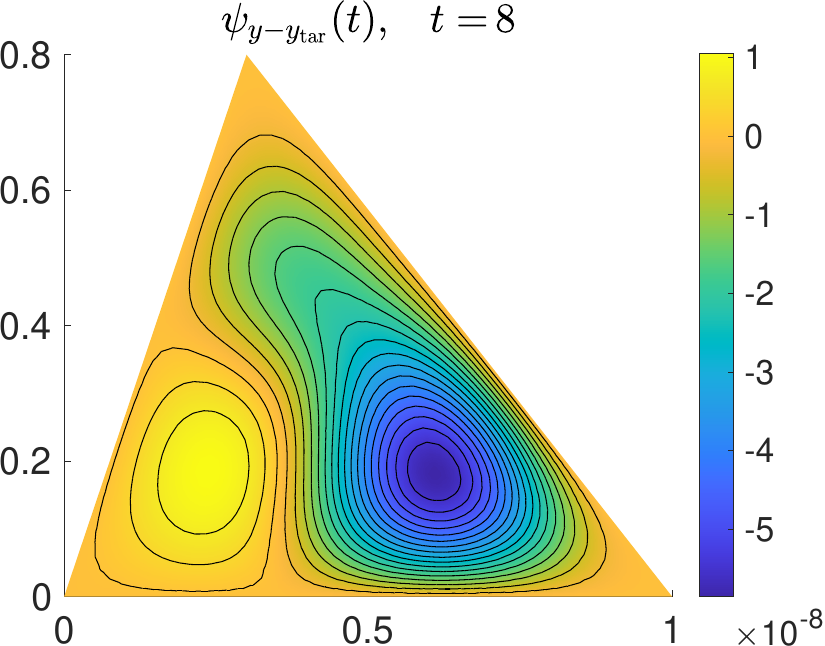}}
\subfigure
{\includegraphics[width=0.32\textwidth,height=0.25\textwidth]{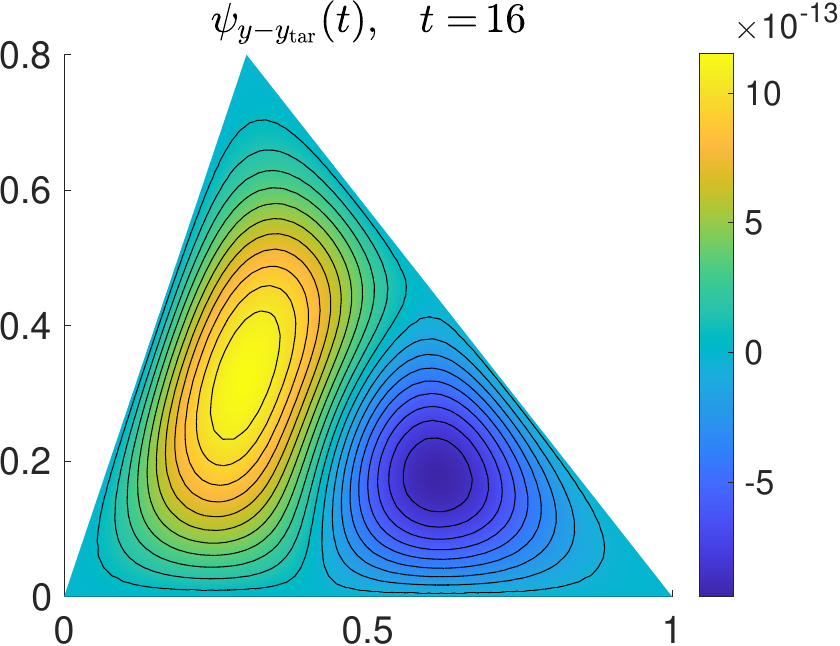}}
\\
\centering
\subfigure
{\includegraphics[width=0.32\textwidth,height=0.25\textwidth]{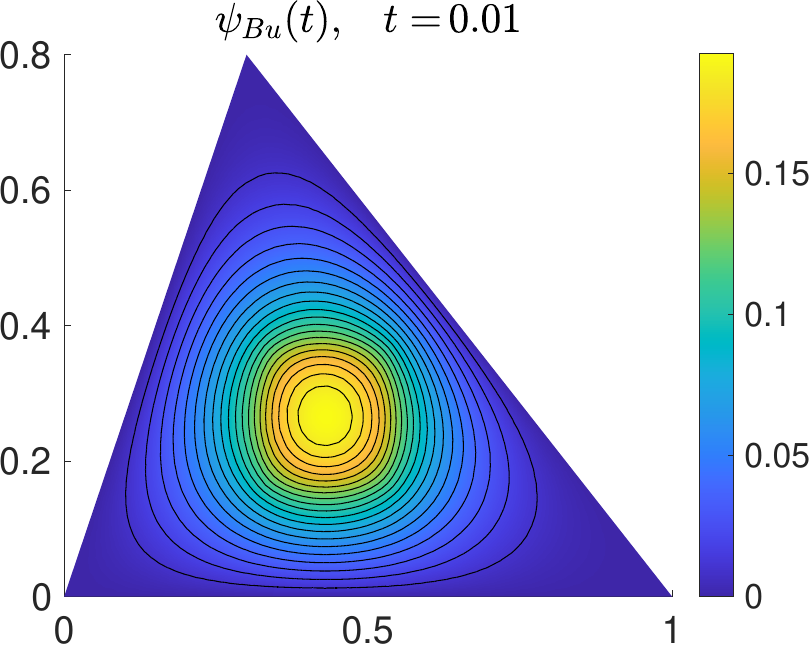}}
\subfigure
{\includegraphics[width=0.32\textwidth,height=0.25\textwidth]{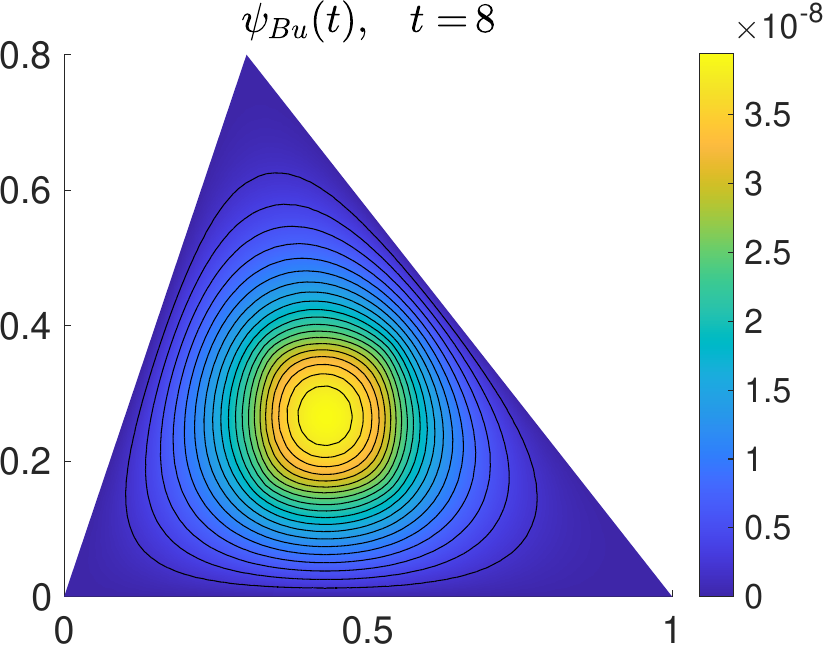}}
\subfigure
{\includegraphics[width=0.32\textwidth,height=0.25\textwidth]{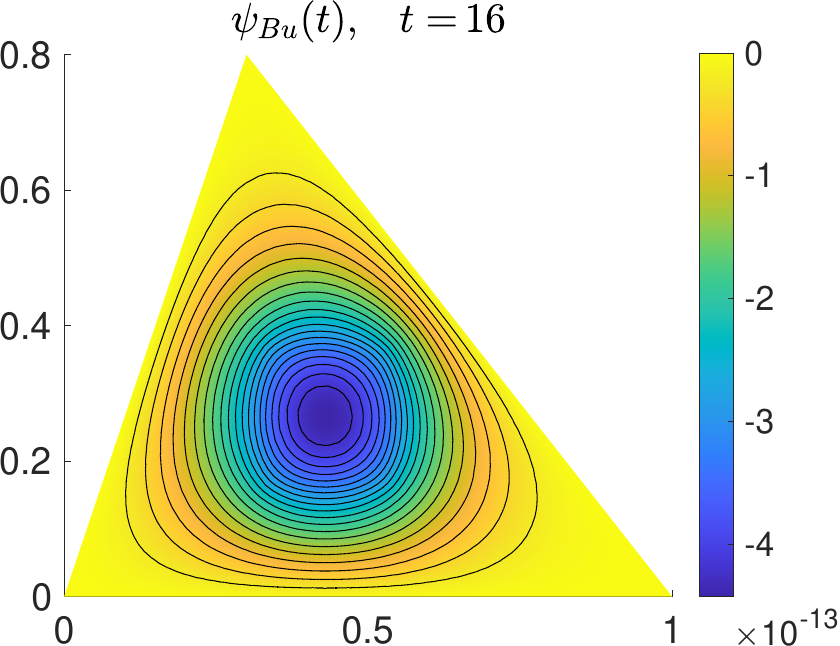}}
\caption{Stream functions of difference to target and control forcing. $M=1$. Data~\eqref{data:unk}.}
\label{fig.snap-unkM1-diffBu}
\end{figure}

Next, in Fig.~\ref{fig.snap-unkM4-diffBu}, we see the time-snapshots of the stream function for the case of~$4$ actuators. In particular, from the stream function of the control forcing we can identify local extrema at the location of the supports of the vorticity of the actuators in Fig.~\ref{fig.mesh-vortM4}.

\begin{figure}[ht]
\centering
\subfigure
{\includegraphics[width=0.32\textwidth,height=0.25\textwidth]{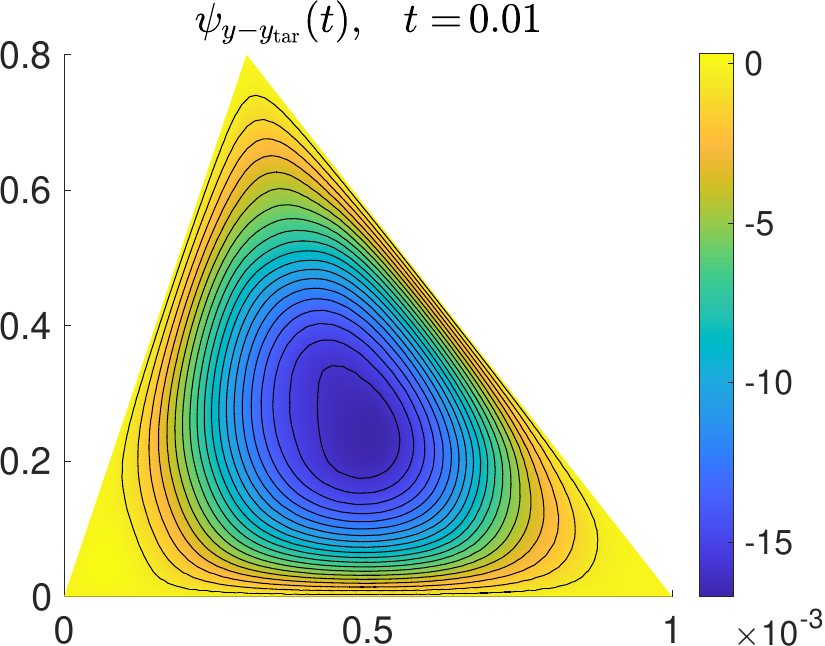}}
\subfigure
{\includegraphics[width=0.32\textwidth,height=0.25\textwidth]{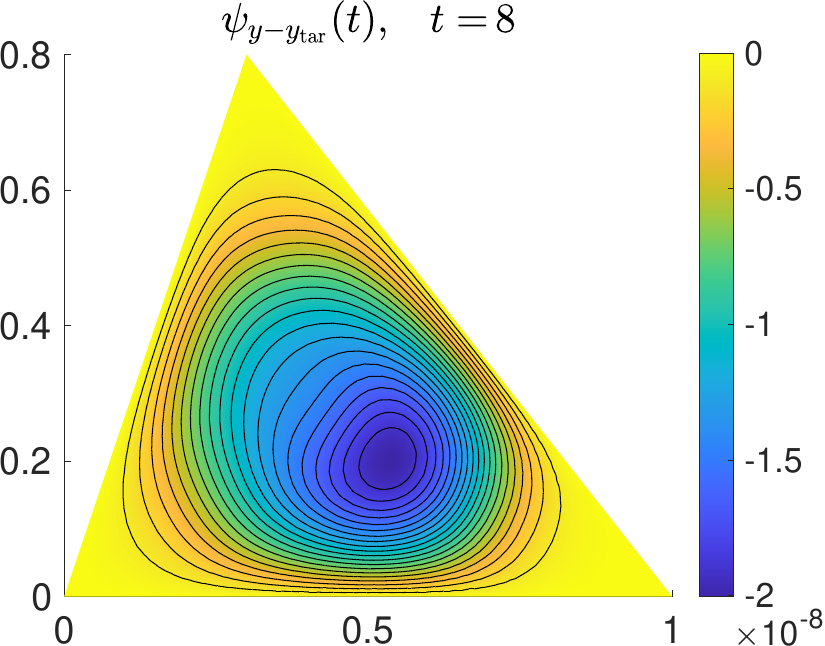}}
\subfigure
{\includegraphics[width=0.32\textwidth,height=0.25\textwidth]{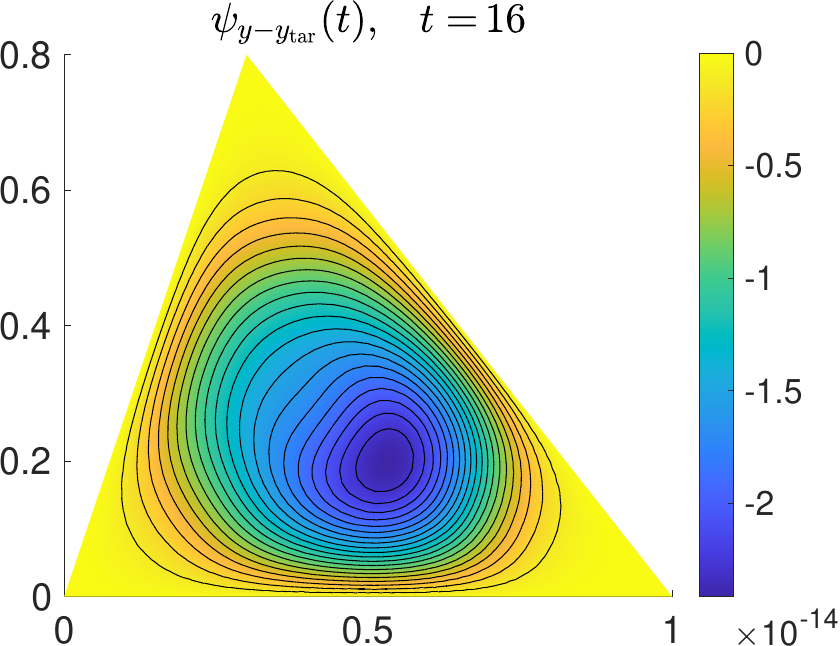}}
\\
\subfigure
{\includegraphics[width=0.32\textwidth,height=0.25\textwidth]{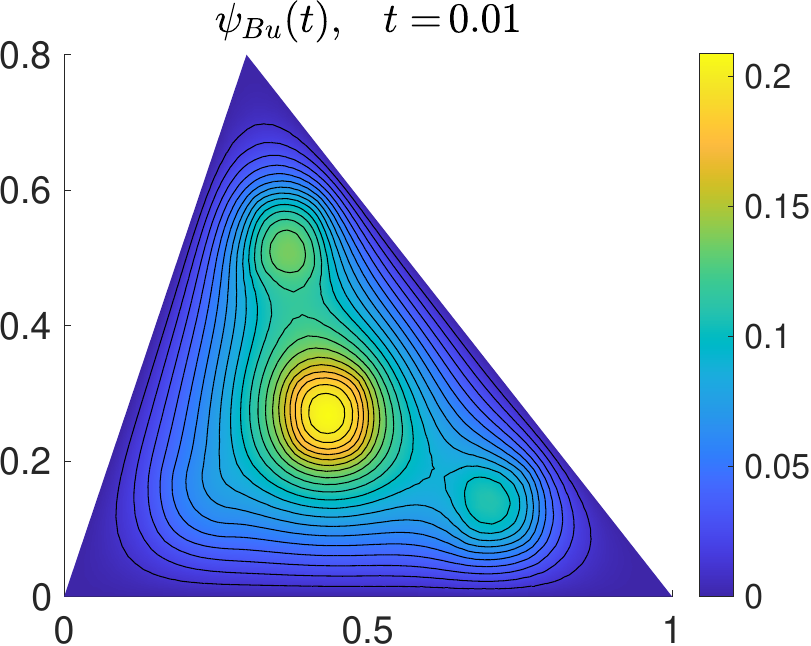}}
\subfigure
{\includegraphics[width=0.32\textwidth,height=0.25\textwidth]{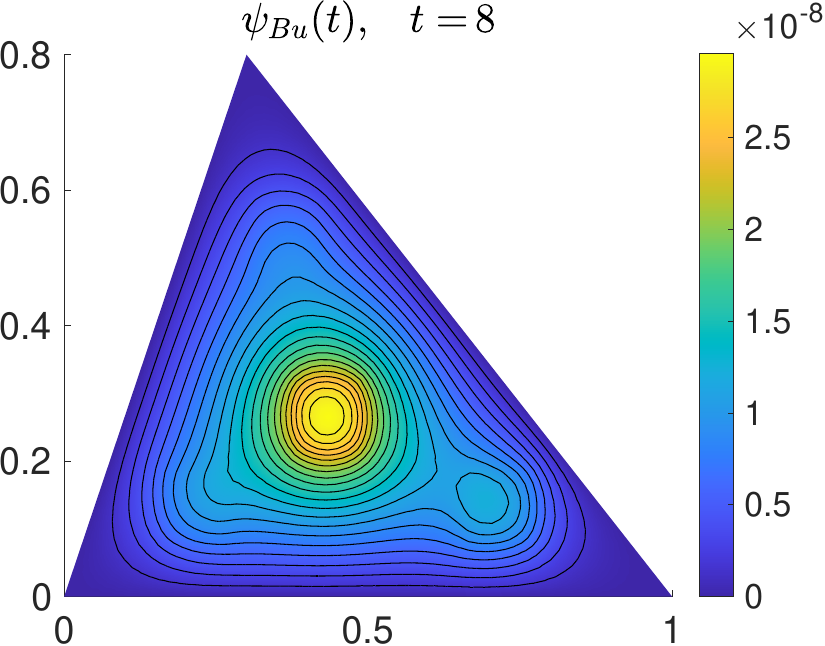}}
\subfigure
{\includegraphics[width=0.32\textwidth,height=0.25\textwidth]{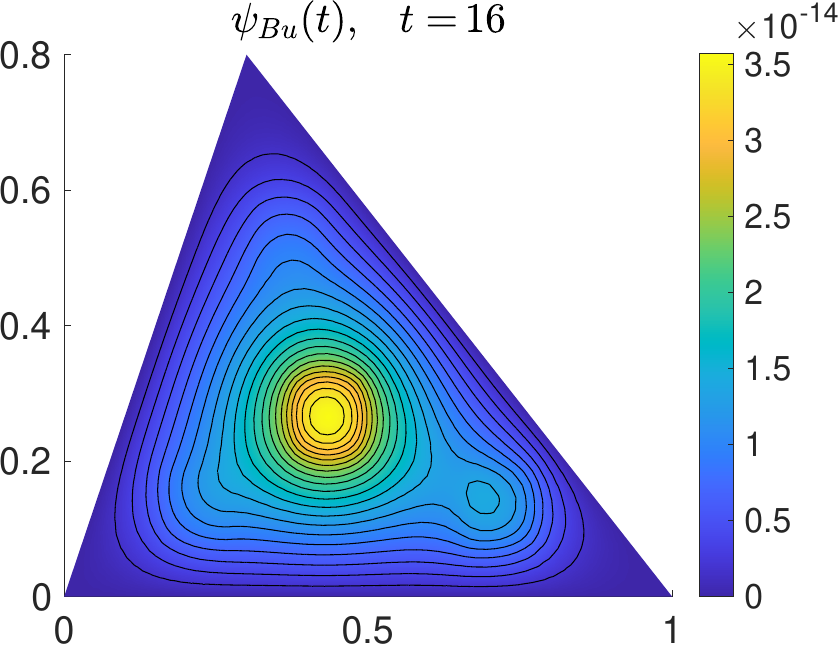}}
\caption{Stream function of difference to target and control forcing. $M=4$. Data~\eqref{data:unk}.}
\label{fig.snap-unkM4-diffBu}
\end{figure}

\end{example}

\begin{example}\label{Ex:exa}
Finally, we take  ad-hoc initial data, in order to make the function
\begin{align}\notag
&w_{\rm exa}(t,x)\coloneqq\sin(2t)(x_1-\tfrac4{10})
\end{align}
the exact solution of the free dynamics. Namely, we take the
external forcings as
\begin{subequations}\label{data:exa}
\begin{align}
&g(t,x)=w_{\rm exa}(t,x)\rest{\p\Omega},\\
&f(t,x)={\dot w_{\rm exa}} -  \Delta w_{\rm exa}  +\nu^{-1}\curl^* (A^{-1}w_{\rm exa})\cdot\nabla w_{\rm exa},
\intertext{and the initial states as}
&w_{\ttt0}(x)=0,\qquad w_0(x)=-10\sin(3x_1)\sin(4x_2)+5.
\end{align}
\end{subequations}

In Fig.~\ref{fig.exa} we confirm again that we can increase the exponential decrease rare of the differente to the target by increasing the number of actuators.

Finally, for the sake of completeness, and since we know the exact analytical expression of the targeted state,  we  compare the numerical solutions to the exact one. In Fig.~\ref{fig.refinements} we consider the~$3$ meshes in~\eqref{mesh-rho}, and compare the computed targeted solution~$w_\ttt$ with the exact one~$w_{\rm exa}$ in Fig.~\ref{fig.tar-exa}. We can see that, in this example, that the solution computed in the coarsest mesh~$\bfR_0$ gives us already a good approximation of~$w_{\rm exa}$. We compare also the computed  controlled solution~$w$ to the exact one~$w_{\rm exa}$ in Fig.~\ref{fig.ctr-exa}. 

Finally, we would like to underline that the exact analytical expression for~$w_{\rm exa}$ was used at the plotting stage only. The control input~$u(t)=K_M^\lambda (w(t)-w_\ttt(t))\in\bbR^{M_\sigma}$ was computed based on the difference between the numerical controlled~$w$ and the numerical targeted~$w_\ttt$ trajectories. That is, in the computations we did not use the fact that we know~$w_{\rm exa}$; we used only the knowledge of the data tuple~$(f,g,w_{\ttt0},w_{0})$.

\begin{figure}[ht]
\centering
\subfigure
{\includegraphics[width=0.45\textwidth,height=0.375\textwidth]{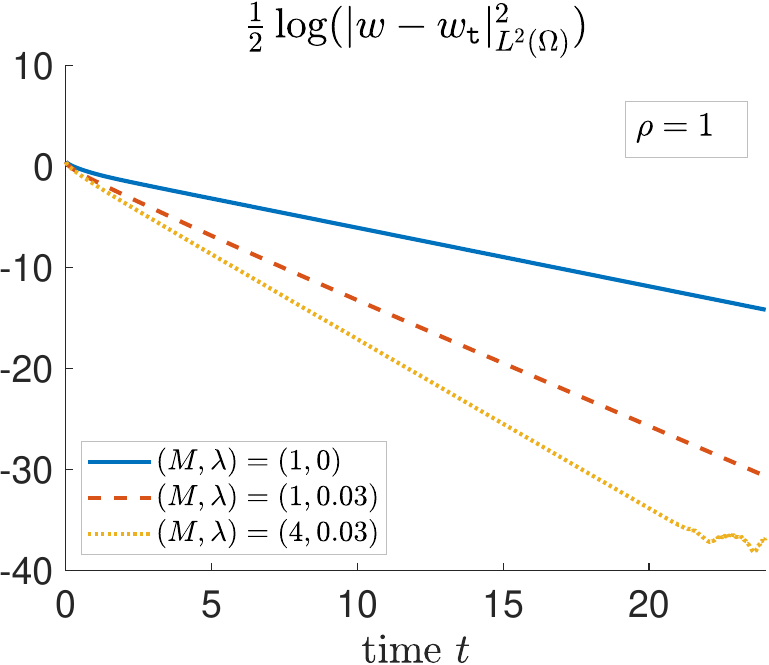}}
\caption{Increasing~$M$. Data~\eqref{data:exa}.}
\label{fig.exa}
\end{figure}

\begin{figure}[ht]
\centering
\subfigure[Numerical error between computed and exact targeted vorticity.\label{fig.tar-exa}]
{\includegraphics[width=0.45\textwidth,height=0.375\textwidth]{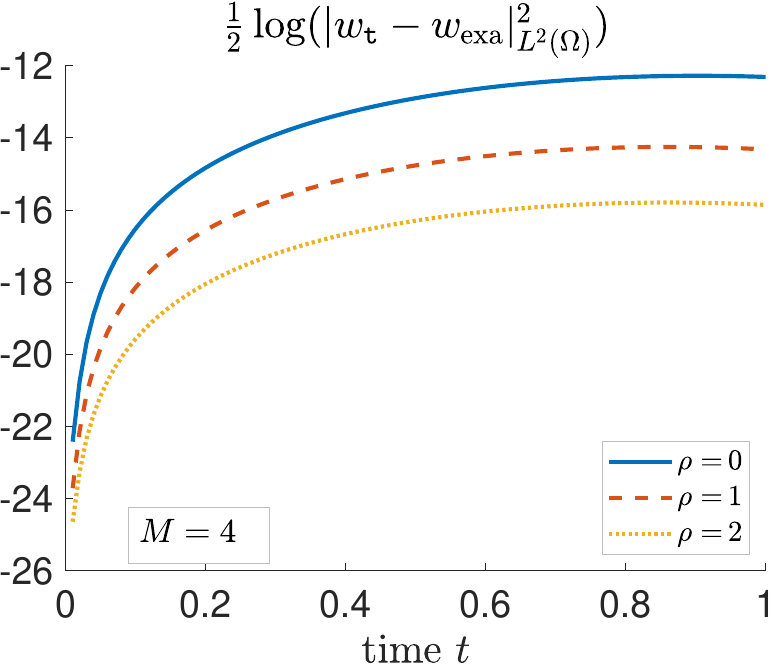}}
\qquad
\subfigure[Difference between computed controlled  vorticity and exact targeted  vorticity.\label{fig.ctr-exa}]
{\includegraphics[width=0.45\textwidth,height=0.375\textwidth]{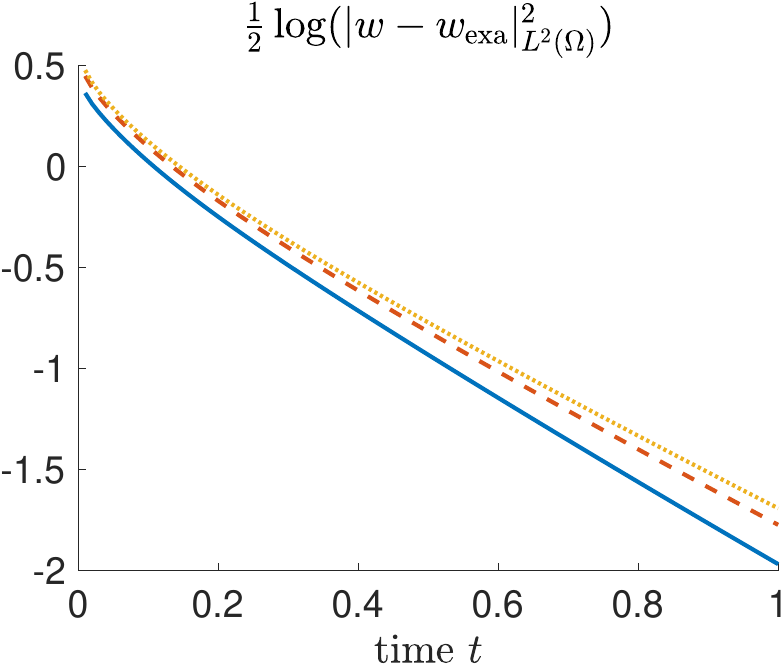}}
\caption{On the refinement level of the discretization. Data~\eqref{data:exa}.}
\label{fig.refinements}
\end{figure}

\end{example}

\begin{remark}\label{R:nonhom-bcs}
We can reduce the case nonhomogeneous boundary conditions to the homogeneous one if~$g$ is regular enough. For example, in case we have~$g= G\rest{\p\Omega}$ with~$G$ smooth in the closed cylinder~$[0,+\infty)\times\overline\Omega$. Then $w^0\coloneqq w-G$ will satisfy the homogeneous boundary conditions and a dynamics analogue to~\eqref{sys-vort-num-dyn} with some different data~$(f,w_{\ttt0})$ and some extra linear terms involving~$G$. The arguments we followed in previous sections can be adapted to this dynamics by taking care of the extra linear terms by including them in the operator~$A_{\rm rc}$ in~\eqref{sys-diff}. We refer the reader to~\cite[Def.~5.2]{Rod14-na}, where such a lifting argument~$g\mapsto G$ has been used.
\end{remark}

\begin{remark}\label{R:freestab}
In both Examples~\ref{Ex:unk} and~\ref{Ex:exa} we see that the free dynamics of the difference to the target is exponentially stable. This suggests that the free dynamics is stable in a neighborhood of the considered targets. We would like to say that we do not know whether this is the case in our setting with Lions boundary conditions in a triangular domains; though, we suspect that it is not. For example, we recall that this is not the case for periodic boundary conditions, as we can see from the works in~\cite{Liu92, Vasudevan21}. From these works we can also see that proving the instability around a given steady state is a nontrivial task, and may require to take~$\nu$ small enough (for a fixed ``normalized'' steady state)~\cite[Thm.~2.9]{Vasudevan21}. Note also that a smaller~$\nu$ requires a finer discretization which means that  numerical computations can become very expensive or unfeasible.  
\end{remark}

\section{Concluding remarks and potential developments} \label{S:conclusion}
We have shown that, given an arbitrary~$\mu>0$, we can find a finite set of actuators~$\Phi_i$ each with vorticity localized in a small domain~$\omega_i$ which enable us to stabilize, with exponential rate~$\mu$, the 2D Navier--Stokes system to a given velocity field trajectory~$y_\ttt$ of the free dynamics. The input control is given by an  explicit linear feedback operator based on suitable oblique projections.

\subsection{On the 3D Navier--Stokes equations}\label{sS:rmks3D}
It is well known that in the case of spatial domains ~$\Omega\subset\bbR^3$ the analysis of the  3D Navier--Stokes equations is more involved.
In fact, the well-posedness of the Cauchy problem on existence, uniqueness, and continuity of the solutions on the initial data is still an open problem, for large time intervals. Another point is that the vorticity is not anymore a scalar function, but a vector
\begin{equation}\notag
w=\Curl y\coloneqq(\curl_{(x_2,x_3)} (y_2,y_3),\curl_{(x_3,x_1)} (y_3,y_1),\curl_{(x_1,x_2)} (y_1,y_2)),
\end{equation}
for a velocity field~$y(x)\eqqcolon(y_1,y_2,y_3)(x)$, with the generalization of~$\curl$ in~\eqref{curl} as
\begin{equation}\notag
\curl_{(a,b)} (f_1,f_2)\coloneqq \tfrac{\p}{\p b}f_1-\tfrac{\p}{\p a}f_2.
\end{equation}
The dynamics of the vorticity is given by a system of three scalar parabolic equations
  \begin{align}\notag
 &\dot w_\ttt  - \nu \Delta w_\ttt  +y_\ttt\cdot\nabla w_\ttt +w_\ttt\cdot\nabla y_\ttt =\curl f,\qquad  
w_\ttt(0)=w_{\ttt0},
\end{align}
instead of the single one~\eqref{sys-vort} that we have the 2D case.  The Lions boundary conditions~$\clG y_\ttt\rest{\p\Omega} = 0$ also take the more cumbersome expression in the 3D case as
\begin{align}\notag
        \clG=(y_\ttt\cdot\bfn, \Curl y_\ttt - ((\Curl y_\ttt)\cdot \bfn)\bfn);
\end{align}
see~\cite{XiaoXin07}. That is, now the vorticity is normal to  the boundary, but not necessarily vanishing as in the 2D case; see~\eqref{lions-bc}.

 Thus, though the possibility of derivation of an analogue stabilizability result seems plausible, it will involve considerable extra work in order to overcome the new regularity issues. In particular, we can see that within the proofs of Lemmas~\ref{LA:A1} and~\ref{LA:NN} we have used arguments (e.g., Sobolev embeddings) which do not hold (exactly in the same way) in the 3D case. Due to the importance of the 3D case in a wider range of real-world applications, it would be interesting to investigate derivation of such an analogue stabilizability result a  future work. 
 
\subsection{On the  ``localized'' actuators}
We have considered actuators with vorticity locally supported in small subdomains with the same shape, up to a translation and a rotation. In the literature we often find actuators taken as solenoidal projections of locally supported vector fields, for example, in the 2D case as the vector fields
\begin{equation}\label{indfvf}
\widehat\Phi_{i_1}^\bullet\coloneqq\indf_{\omega_{i}}(1,0),\qquad\widehat\Phi_{i_2}^\bullet\coloneqq\indf_{\omega_{i}}(0,1),\qquad1\le i\le M;
\end{equation}
supported in subdomains~$\omega_i\subset\Omega$,~$\omega_i\ne\Omega$,  of the spatial domain~$\Omega$; see~\cite[Sect.~4, Exa.~2]{Azmi22}. See also~\cite[Sect.~4, Exa.~2]{Azmi22}, \cite[Eqs.~(1.1) and~(1.4)]{BarbuTri04} and~\cite[Sect.~3]{BarRodShi11},  where
other locally supported vector fields are taken. 

Actuators as in~\eqref{indfvf} are in~$(L^2(\Omega))^2$, but are not in the state space~$\bfH$ of solenoidal vector fields and their Leray projections~$\widehat\Phi_{i_j}\coloneqq\Pi\widehat\Phi_{i_j}^\bullet$ onto~$\bfH$ depends on~$\Omega$ (because~$\Pi$ depends on~$\Omega$).
Analogously, for the actuators~$\Phi_i$ considered in this manuscript, say with localized vorticity~$\curl\Phi_i=\indf_{\omega_i}$, we can see that
~$\Phi_i=\curl^*A^{-1}\indf_{\omega_i}$ also depends on~$\Omega$ (because~$A$ depends on~$\Omega$). Thus, neither of the actuators are localized, in the sense that neither~$\widehat\Phi_{i_j}$ nor~$\Phi_{i}$ can be constructed independently of~$\Omega$.
This observation raises the following question, which could be important for applications:
 can we find a family of stabilizing actuators, in~$\bfH$ and supported in small subdomains, which can be constructed (manufactured) independently of~$\Omega$? The investigation of this question is an interesting subject for future work.

\subsection{On output based feedback stabilization}
We have seen in Section~\ref{sS:intro-main-res-data} that the feedback control system that we propose can be interpreted  as a Luenberger observer, thus our strategy can be applied to state estimation (also known as continuous data assimilation~\cite{AzouaniOlsonTiti14}). We can construct a set of sensors and a set of actuators as a slight variation of the construction in Section~\ref{sS:act}. Namely, putting a sensor besides an actuator instead of just an actuator/sensor as in Figs.~\ref{fig.suppActRect} and~\ref{fig.suppActTri}. This has been done in~\cite[Fig.~1]{Rod21-aut}, where a feedback control was coupled with an observer to achieve an output-based stabilization result in the case of linear dynamics. The extension of such a result for general nonlinear dynamics is a nontrivial problem, because the so-called separation principle does not hold, in general. However, the investigation of this problem is of paramount importance for real world applications. In particular, the derivation of results on this direction concerning the nonlinear Navier--Stokes equations is an interesting subject for future research.

\bigskip\noindent
{\bf Aknowlegments.}
D. Seifu was supported by the State of Upper Austria and the Austrian Science
Fund (FWF): P 33432-NBL,  S. Rodrigues acknowledges partial support from the same grant.

{\small%
 \bibliographystyle{plainurl}%
  \bibliography{NS_OProj}%
}

\end{document}

%% file: Mathcommands.tex
\newcommand{\linspan}{\mathop{\rm span}\nolimits}

\newcommand{\supp}{\mathop{\rm supp}\nolimits}
\newcommand{\diver}{\mathop{\rm div}\nolimits}
\newcommand{\curl}{\mathop{\rm curl}\nolimits}
\newcommand{\Curl}{\mathop{\rm Curl}\nolimits}

\newcommand{\sign}{\mathop{\rm sign}\nolimits}

\newcommand{\rest}{\left.\kern-2\nulldelimiterspace\right|_}
\newcommand{\norm}[2]{\left|#1\right|_{#2}}

\newcommand{\vol}{\mathop{\rm vol}\nolimits}

\newcommand{\indf}{1}

\newcommand{\ex}{\mathrm{e}}
\newcommand{\p}{\partial}
\newcommand{\ed}{\mathrm d}

\newcommand*{\Bigcdot}{\raisebox{-.25ex}{\scalebox{1.25}{$\cdot$}}}

\newcommand{\clC}{{\mathcal C}}

\newcommand{\clF}{{\mathcal F}}
\newcommand{\clG}{{\mathcal G}}
\newcommand{\clH}{{\mathcal H}}

\newcommand{\clK}{{\mathcal K}}
\newcommand{\clL}{{\mathcal L}}

\newcommand{\clO}{{\mathcal O}}
\newcommand{\clP}{{\mathcal P}}

\newcommand{\clT}{{\mathcal T}}
\newcommand{\clU}{{\mathcal U}}
\newcommand{\clV}{{\mathcal V}}

\newcommand{\bbN}{{\mathbb N}}

\newcommand{\bbR}{{\mathbb R}}

\newcommand{\bfA}{{\mathbf A}}

\newcommand{\bfF}{{\mathbf F}}
\newcommand{\bfG}{{\mathbf G}}
\newcommand{\bfH}{{\mathbf H}}

\newcommand{\bfJ}{{\mathbf J}}
\newcommand{\bfK}{{\mathbf K}}

\newcommand{\bfP}{{\mathbf P}}

\newcommand{\bfR}{{\mathbf R}}

\newcommand{\bfV}{{\mathbf V}}

\newcommand{\fkB}{{\mathfrak B}}

\newcommand{\fkN}{{\mathfrak N}}

\newcommand{\rmD}{{\mathrm D}}

\newcommand{\bfj}{{\mathbf j}}

\newcommand{\bfn}{{\mathbf n}}

\newcommand{\bft}{{\mathbf t}}

\newcommand{\rmd}{{\mathrm d}}
\newcommand{\rme}{{\mathrm e}}

\newcommand{\ttt}{{\tt t}}

\definecolor{DarkBlue}{rgb}{0,0.08,0.45}
\definecolor{DarkRed}{rgb}{.65,0,0}
\definecolor{applegreen}{rgb}{0.55, 0.71, 0.0}

\newcounter{mymac@matlab}
\newcommand{\matlab}{MATLAB%
   \ifnum\value{mymac@matlab}<1%
   \textregistered%
   \setcounter{mymac@matlab}{1}%
   \fi%
  }
